\documentclass[10pt,reqno]{amsart}
\usepackage[a4paper,top=2.5cm,right=3.3cm,left=3.3cm,bottom=2.5cm,footskip=20pt]{geometry}
\usepackage[hidelinks]{hyperref} %backref=page
\pdfoutput=1
\usepackage{amssymb,amsaddr,eulervm,palatino,enumitem,float,microtype}
\usepackage[USenglish]{babel}
\usepackage{mathtools}
\mathtoolsset{showonlyrefs=true}
\usepackage{tikz}
\usetikzlibrary{arrows,arrows.meta,positioning,calc}
\setenumerate{label=\textup{(\roman*)},itemsep=2pt,topsep=2pt,leftmargin=2.5em}
\usepackage{abstract}

\tikzset{
>=stealth,
node distance=2cm,
main node/.style={circle,inner sep=1.5pt,fill=black},
freccia/.style={->,shorten >=2pt,shorten <=2pt,semithick},
ciclo/.style={out=50, in=130, loop, distance=2cm, ->,semithick}}

\allowdisplaybreaks[4]
\newcommand{\doi}[1]{\url{https://doi.org/#1}}
\newcommand{\isbn}[1]{\url{https://isbnsearch.org/isbn/#1}}
\newcommand{\arxiv}[1]{\href{https://arxiv.org/abs/#1}{preprint arXiv:#1}}
\newcommand{\web}[1]{\url{#1}}

\renewcommand{\emptyset}{\varnothing}
\numberwithin{equation}{section}

\newtheorem{thm}{Theorem}[section]
\newtheorem{lemma}[thm]{Lemma}
\newtheorem{prop}[thm]{Proposition}
\newtheorem{rem}[thm]{Remark}
\newtheorem{cor}[thm]{Corollary}
\newtheorem{df}[thm]{Definition}
\newtheorem{ex}[thm]{Example}

\newcommand{\C}{\mathbb{C}}
\newcommand{\Z}{\mathbb{Z}}
\newcommand{\N}{\mathbb{N}}
\newcommand{\id}{\mathrm{id}}
\newcommand{\mat}[1]{\bigg(\!\begin{array}{cc}#1\end{array}\!\bigg)}

\makeatletter
\def\@tocline#1#2#3#4#5#6#7{\relax
  \ifnum #1>\c@tocdepth % then omit
  \else
    \par \addpenalty\@secpenalty\addvspace{#2}%
    \begingroup \hyphenpenalty\@M
    \@ifempty{#4}{%
      \@tempdima\csname r@tocindent\number#1\endcsname\relax
    }{%
      \@tempdima#4\relax
    }%
    \parindent\z@ \leftskip#3\relax \advance\leftskip\@tempdima\relax
    \rightskip\@pnumwidth plus4em \parfillskip-\@pnumwidth
    #5\leavevmode\hskip-\@tempdima
      \ifcase #1
       \or\or \hskip 1em \or \hskip 2em \else \hskip 3em \fi%
      #6 \hskip 0.5em \nobreak\relax
    \dotfill\hbox to\@pnumwidth{\@tocpagenum{#7}}\par
    \nobreak
    \endgroup
  \fi}
\makeatother

\begin{document}

\title{\vspace*{-1cm}On the K-theory of the AF core of a graph C*-algebra}

\author[F.~D'Andrea]{\vspace*{-5mm}Francesco D'Andrea}

\address{Dipartimento di Matematica e Applicazioni ``R.~Caccioppoli'' \\ Universit\`a di Napoli Federico II \\
Complesso MSA, Via Cintia, 80126 Napoli, Italy}



\maketitle

\begin{abstract}
\vspace{3pt}

In this paper, we study multiplicative structures on the K-theory of the core $A$ of the C*-algebra of a directed graph $E$. In the first part of the paper, we study embeddings $E\to E\times E$ that induce a *-homomorphism $A\otimes A\to A$.
Through the K{\"u}nneth formula, any such a *-homomorphism induces a ring structure on $K_*(A)$.
In the second part, we give conditions on $E$ such that $K_*(A)$ is generated by ``noncommutative line bundles'' (invertible bimodules). The same conditions guarantee the existence of a homomorphism of abelian groups $K_0(A)\to\Z[\lambda]/(\det(\lambda\Gamma-1))$ (where $\Gamma$ is the adjacency matrix of $E$) that is compatible with the tensor product of line bundles. Examples include the C*-algebra $C(\C P^{n-1}_q)$ of a quantum projective space, the $UHF(n^\infty)$ algebra, and the C*-algebra of the space parameterizing Penrose tilings. For the first algebra, as a corollary we recover some identities that classically follow from the ring structure of $K^0(\C P^{n-1})$, and that were proved by Arici, Brain and Landi in the quantum case.
Incidentally, we observe that the C*-algebra of Penrose tilings is the AF core of the Cuntz algebra $\mathcal{O}_2$, if the latter is realized using the appropriate graph.
\end{abstract}

\vspace{7mm}

\begin{center}
\begin{minipage}{0.8\textwidth}
\parskip=0pt\small\tableofcontents
\end{minipage}
\end{center}

\newpage

\section{Introduction}

A crucial difference between the K-theory of a topological space and that of a C*-algebra is that the former is a ring, with product induced by the tensor product of vector bundles, while the latter is, in general, only an abelian group.

A classical results by Atiyah and Todd states that, for all $n\geq 2$,
\begin{equation}\label{eq:TPol}
K^0(\C P^{n-1})=\Z[x]/(x^n)
\end{equation}
is a ring of truncated polynomials \cite{AT60}. In the above formula, $x:=1-[L_1]$ is the Euler class of the dual $L_1$ of the tautological line bundle (see e.g.~\cite[Cor.~2.8, page 291]{Kar78}).
Notice that $K^*(\C P^{n-1})$ and $K^*$ of the discrete space with $n$ points are isomorphic as $\Z/2$-graded abelian groups, but not as rings. This is a very simple and naive example explaining the importance of the ring structure of topological K-theory. (In this example, as well as in the case of AF algebras, $K^1$ is zero and we will often forget about it.)

Let $L_k$ be the tensor product of $k\geq 0$ copies of $L_1$, with the convention that $L_0$ is the trivial line bundle (so, $[L_0]=1$ is the unit of the K-ring).
It follows from \eqref{eq:TPol} that the classes of line bundles
\begin{equation}\label{eq:Kbasis}
\big\{ [L_k]:0\leq k\leq n-1 \big\}
\end{equation}
form a basis of $K^0(\C P^{n-1})$ as a $\Z$-module. Using the binomial formula and $[L_1]^k=[L_k]$ for $k\in\Z$, we can rewrite $x^n=0$ as the following identity:
\begin{equation}\label{eq:AT1}
\sum_{k=0}^n\binom{n}{k}(-1)^k[L_k]=0 .
\end{equation}
Following \cite{DHMSZ20}, we will refer to \eqref{eq:AT1} as the \emph{Atiyah-Todd identity} for $\C P^{n-1}$.

\smallskip

Despite the lack of ring structure, for the C*-algebra $C(\C P^{n-1}_q)$ of a quantum projective space, Arici, Brain and Landi in \cite{ABL14} proved a result which is analogous to \eqref{eq:AT1}. In this context, compact Hausdorff spaces are replaced by unital C*-algebras and line bundles are replaced by self-Morita equivalence bimodules. One starts with the C*-algebra of a Vaksman-Soibelman quantum sphere $C(\mathbb{S}^{2n-1}_q)$ \cite{VS91}.
This algebra is strongly $\Z$-graded, with degree $0$ part given by $A:=C(\C P^{n-1}_q)$ and degree $1$ part given by an $A$-bimodule $\mathcal{L}_1$ which is the analogue of the bimodule of continuous sections of the line bundle $L_1$. It was shown in \cite{ABL14}, using an index argument, that powers $\mathcal{L}_k:=(\mathcal{L}_1)^{\otimes_Ak}$
of this bimodule satisfy the identity \eqref{eq:AT1} (here $\otimes_A$ denotes the balanced tensor product of $A$-bimodules, and by $[\mathcal{L}_k]$ we mean the K-theory class of $\mathcal{L}_k$ as a left $A$-module).
Moreover, it is known that $K_0(A)$ has a basis, as a $\Z$-module, given by the classes of $\mathcal{L}_0,\mathcal{L}_1,\ldots,\mathcal{L}_{n-1}$ (see e.g.~\cite{DL10}). If there were a ring structure on $K_0(A)$ induced by the balanced tensor product of bimodules, called $x:=1-[\mathcal{L}_1]$ like in the classical case, then we could interpret \eqref{eq:AT1} by arguing that $K_0(A)$ is isomorphic to \eqref{eq:TPol} as a ring.

\smallskip

The naive idea would be, given a C*-algebra $A$, to represent the K-theory class of a finitely generated projective left $A$-module by a bimodule, and for any two such bimodules $M$ and $N$ to define a product in K-theory by
\begin{center}
`` $[M]\cdot [N]:=[M\otimes_AN]$ ``.
\end{center}
This idea, unfortunately, cannot work. Isomorphism of bimodules is stronger than isomorphism of left modules, and the same $K_0$-class might be represented by non-isomorphic bimodules, so that the above operation is in general ill-defined. In addition, not every K-theory class can be represented by a bimodule. A simple counterexample is the left $M_n(\C)$-module $\C^n$, which has no $M_n(\C)$-bimodule structure if $n\geq 2$.

The issue of possible ring structures on the K-theory of a noncommutative C*-algebra is addressed in \cite{DM22}, where we construct a multiplicative K-theory functor using finitely generated projective modules associated to Hopf-Galois extensions of a C*-algebra. We will not discuss this approach here.

\medskip

The first part of the present paper is about some attempts to find a *-homomorphism $C(\C P^{n-1}_q)\otimes C(\C P^{n-1}_q)\to C(\C P^{n-1}_q)$ inducing the ring structure \eqref{eq:TPol}.
Our story begins with the intersection product of KK-theory (\textit{aka} external Kasparov product).
For any separable C*-algebra $A$, it gives a morphism of abelian groups 
\begin{equation}\label{eq:1}
K_*(A)\otimes K_*(A)\to K_*(A\otimes A) ,
\end{equation}
where the one on the right hand side is the minimal tensor product of C*-algebras.
On the other hand, any *-homomorphism
\begin{equation}\label{eq:2}
\psi:A\otimes A\to A
\end{equation}
induces a morphism $K_*(A\otimes A)\to K_*(A)$. Composed with \eqref{eq:1}, this gives a morphism of abelian groups $K_*(A)\otimes K_*(A)\to K_*(A)$ that transforms $K_*(A)$ into a ring.
A natural candidate for $\psi$ would be the multiplication map, but this is a *-homomorphism if and only if $A$ is commutative.
It is then an interesting problem, for non-commutative C*-algebras $A$, to find a *-homomorphism (if any) of the form \eqref{eq:2}.

A remarkable property of our main example $C(\C P^{n-1}_q)$ is that it is an AF algebra. Note that if $A$ is an AF-algebra, then $K_*(A)$ has the structure of an ordered abelian group. It is known, in this case, that if there is a multiplication on $K_*(A)$ that gives it the structure of an ordered ring, then such a multiplication map must be induced by a (non-explicit) *-homomorphism of the form \eqref{eq:2}, cf.~\cite{AE20}.

In the spirit of Noncommutative Geometry \cite{Con94}, we will think of a (unital) C*-algebra as describing some virtual (compact) quantum space.

For a compact Hausdorff space $X$, the canonical ring structure of $K^*(X)$ is induced by the multiplication map 
$C(X)\otimes C(X)\to C(X)$, that we can think as the composition of the isomorphism
\begin{equation}\label{eq:XxX}
C(X)\otimes C(X)\cong C(X\times X)
\end{equation}
with the pullback
\begin{equation}\label{eq:multimap}
C(X\times X)\to C(X)
\end{equation}
of the diagonal embedding $X\to X\times X$.
We would like to find a similar construction for a nice class of compact quantum spaces. This class should include the example
we are interested in, namely $C(\C P^{n-1}_q)$, which is both the AF core of a graph C*-algebra and a graph C*-algebra itself \cite{HS02}.
The idea is, then, to work within the framework of graph C*-algebras.

We need three ingredients: an isomorphism similar to \eqref{eq:XxX}, where the Cartesian product of spaces is replaced by the Cartesian product of graphs; an embedding of graphs similar to the diagonal embedding of spaces; and a framework where the association of a graph to its graph C*-algebra gives rise to a contravariant functor, so that from the graph embedding we may pass to a *-homomorphism similar to \eqref{eq:multimap}.

In general, if $E$ is a directed graph, $C^*(E)\otimes C^*(E)$ is not isomorphic to $C^*(E\times E)$ (in fact, the class of graph C*-algebras is not even closed under tensor product).
Nevertheless, if $E$ is a row-finite graph with no sinks, there is an injective *-homomorphism $C^*(E\times E)\to C^*(E)\otimes C^*(E)$ which induces an isomorphism of AF cores:
\begin{equation}\label{eq:7}
C^*(E)^{U(1)}\otimes C^*(E)^{U(1)}\cong C^*(E\times E)^{U(1)} .
\end{equation}
This takes care of the first part of the construction in \eqref{eq:XxX}.

Next, for any directed graph $E$, there is a diagonal embedding
\begin{equation}\label{eq:6b}
E\to E\times E .
\end{equation}
Unfortunately, the association $E\to C^*(E)$ sending a graph to its graph C*-algebra doesn't give rise to a contravariant functor (at least not in a naive way, see e.g.~\cite{HT24}), and here is where our dreams crash and burn.
We have a functor (to $U(1)$-C*-algebras) only if we restrict the graph morphisms to a class of ``admissible'' ones.
It is then natural to wonder for what class of graphs the diagonal embedding is admissible. It turns out the diagonal embedding is admissible if and only if the transposed graph is a functional graph. A property of such graphs is that their AF core is commutative (it is the so-called diagonal subalgebra).

There is another embedding of the form \eqref{eq:6b} which one can consider, and is the analogue of the vertical embedding $X\to X\times X$, $x\mapsto (p,x)$, where $p\in X$ is any fixed point. This embedding is defined and admissible for a class of graphs which includes the one of $\mathbb{S}^{2n-1}_q$, and induces a ring structure on the K-theory of $\C P^{n-1}_q$.
The bad news is that this ring structure is not what we want: it is far from \eqref{eq:TPol} for quantum projective spaces, so in particular it doesn't give a geometrical explanation to the origin of the Atiyah-Todd identity for $\C P^{n-1}_q$.
To understand what this ring structure is, one can think about the example of a vertical embedding of a compact Hausdorff space $X$. The induced map $\mu:K^*(X)\otimes K^*(X)\to K^*(X)$ can be easily computed and is given by $\mu([V_1],[V_2])=\mathrm{rk}(V_1)\cdot [V_2]$ for any pair of vector bundles $V_1$ and $V_2$ on $X$, where $\mathrm{rk}$ denotes the rank.
This map is independent of the choice of base point and is, in general, non-commutative.

\medskip

In the second part of the paper, we study a semiring $\mathsf{SPic}(A)$ of classes of invertible $A$-bimodules, that we interpret as describing some sort of ``noncommutative line bundle'' (see Sect.~\ref{sec:lineAF} for the exact definition). 
There is a natural homomorphism 
\begin{equation}\label{eq:SpicK0}
\mathsf{SPic}(A)\to K_0(A)
\end{equation}
of abelian semigroups. For $A=C^*(E)^{U(1)}$ we will show that, if $E$ is finite and the adjacency matrix $\Gamma$ of $E$ is invertible over the integers, then
\[
K_0\big(C^*(E)^{U(1)}\big)\cong\Z^{|E^0|}
\]
has a basis consisting of the classes of the vertex projections of $E$.
Under the same assumptions one can write down explicit identities, involving the K-theory classes of line bundles, that generalize the Atiyah-Todd identity of $\C P^{n-1}$
(Prop.~\ref{cor:AT}). This gives an independent proof of the identity \eqref{eq:AT1} for $\C P^{n-1}_q$, that immediately follows as an easy corollary.
We also show that, when an additional algebraic condition is satisfied by $\Gamma$,
then $K_0\big(C^*(E)^{U(1)}\big)$ is isomorphic to the abelian groups underlying the ring
\begin{equation}\label{eq:multmor}
\Z[\lambda]/\bigl(\det (\lambda\Gamma-1) \bigr) ,
\end{equation}
and the composition of this isomorphism with \eqref{eq:SpicK0} gives a homomorphism of semirings. A nice geometrical interpretation comes from KK-theory, where we can identify \eqref{eq:multmor} with a subring of $KK\big(C^*(E)^{U(1)},C^*(E)^{U(1)}\big)$ generated by the adjacenty matrix $\Gamma$. Examples include quantum complex projective spaces and the C*-algebra parameterizing Penrose tilings. The same theorem can be proved for the $UHF(n^\infty)$ algebra, 
the AF core of the Cuntz C*-algebra $\mathcal{O}_n$,
but it is proved separately since in this case the adjacency matrix is not invertible over the integers.
Incidentally, we notice that the C*-algebra parameterizing Penrose tilings is also the AF core of Cuntz C*-algebra $\mathcal{O}_2$, but for a different choice of directed graph.

\medskip

The paper is organized as follows.
In Sect.~\ref{sec:2} we recall some basic definitions and properties of graph C*-algebras and fix the notations that we use throughout the paper.
In Sect.~\ref{sec:3} we study the C*-algebra of a Cartesian product of graphs and prove \eqref{eq:7}.
In Sect.~\ref{sec:4} we discuss morphisms of graphs, review the notion of admissible morphism from \cite{HT24} and present some examples.
In Sect.~\ref{sec:5} we identify two classes of graphs, one for which the diagonal morphism is admissible, and another for which a vertical morphism is admissible. The latter class includes the graphs of odd-dimensional quantum spheres.
In Sect.~\ref{sec:lineAF} we introduce the semiring $\mathsf{SPic}(A)$ for $A=C^*(E)^{U(1)}$, find conditions that guarantee that $K_0(A)$ is generated by vertex projections of $E$, and prove some formulas relating line bundles and vertex projections.
In Sect.~\ref{sec:appli} we prove the Atiyah-Todd identities for a class of AF cores and discuss the ``multiplicative'' map \eqref{eq:multmor}.
In Sect.~\ref{sec:7} we comment on the ring $KK(A,A)$ for $A=C^*(E)^{U(1)}$.
In App.~\ref{sec:app} we collect some remarks about the Picard group of a noncommutative algebra, with a special care of finite-dimensional C*-algebras.

\bigskip

\begin{center}
\textsc{Acknowledgements}
\end{center}

\noindent
I would like to thank Giovanni Landi and Piotr M.~Hajac for their comments on a preliminary version of this paper.
This research is part of the EU Staff Exchange project 101086394 ``Operator Algebras That One Can See''. It was partially supported by the University of Warsaw Thematic Research Programme ``Quantum Symmetries''.
The Author is a member of INdAM-GNSAGA (Istituto Nazionale di Alta Matematica) -- Unit\`a di Napoli and of INFN -- Sezione di Napoli.

\medskip

\section{Generalities about graph C*-algebras}\label{sec:2}

Let us recall the definition and some basic properties of graph C*-algebras.
Our main references are \cite{BPRS00,FLR00,R05}. We adopt the conventions of \cite{BPRS00,FLR00}, i.e.~the roles of source and range maps are exchanged with respect to~\cite{R05}.

A directed graph $E=(E^0,E^1,s,r)$ consists of a set $E^0$ of vertices, a set $E^1$
of edges, and range and source maps $r,s:E^1\to E^0$. (Sometimes the word ``directed'' will be omitted, but by a graph in this paper we shall always mean a directed graph.)
If $e\in E^1$, $v:=s(e)$ and $w:=r(e)$ we say that $v$ \emph{emits} $e$ and that
$w$ \emph{receives} $e$.

A vertex $v\in E^0$ is called
\begin{itemize}
\item a \emph{sink} if it has no outgoing edges, i.e.~$s^{-1}(v)=\emptyset$,
\item an \emph{infinite emitter} if $|s^{-1}(v)|=\infty$,
\item a \emph{source} if it has no incoming edges, i.e.~$r^{-1}(v)=\emptyset$,
\item \emph{singular} if it is either a sink or an infinite emitter,
\item \emph{regular} if it is not singular, i.e.~if $s^{-1}(v)$ is a finite non-empty set.
\end{itemize}
We denote by $E^0_{\mathrm{reg}}$ the set of regular vertices.
A graph $E$ is called \emph{row-finite} if it has no infinite emitters, and \emph{finite} if both $E^0$ and $E^1$ are finite sets (thus, a finite graph is a row-finite graph with finitely many vertices).

A Cuntz-Krieger $E$-family for an arbitrary graph $E$ consists of mutually orthogonal projections $\{P_v:v\in E^0\}$ and partial isometries $\{S_e:e\in E^1\}$ with orthogonal ranges satisfying the following relations \cite{FLR00}
\begin{alignat}{4}
S^*_eS_e &= P_{r(e)} && \forall\;e\in E^1 \tag{CK1}\label{eq:CK1} \\
S_eS_e^* &\leq P_{s(e)} && \forall\;e\in E^1 \tag{CK2}\label{eq:CK2} \\
\sum_{e\in s^{-1}(v)}\!\! \!S_eS_e^* &=P_v && \forall\;v\in E^0_{\mathrm{reg}},  \tag{CK3}\label{eq:CK3}
\end{alignat}
Note that \eqref{eq:CK2} follows from \eqref{eq:CK3} if $s(e)$ is a regular vertex.

The graph C*-algebra $C^*(E)$ of a graph $E$ is defined as the universal C*-algebra generated by a Cuntz-Krieger $E$-family, and is unique up to an isomorphism which sends Cuntz-Krieger generators into Cuntz-Krieger generators, see \cite{BPRS00,FLR00,R05}.

\medskip

The \emph{gauge action} \mbox{$\gamma:U(1)\to\mathrm{Aut}\,C^*(E)$} is defined on generators by $\gamma_t(S_e)=t\,S_e$ and $\gamma_t(P_v)=P_v$ for all $t\in U(1),e\in E^1,v\in E^0$.
A subset $I\subset C^*(E)$ (e.g., a two-sided *-ideal) is called \emph{gauge-invariant} if $\gamma_t(I)\subset I\;\forall\;z\in U(1)$.
The gauge action defines a $\Z$-grading of $C^*(E)$: called
\begin{equation}\label{eq:Lk}
\mathcal{L}_k:=\big\{a\in C^*(E):\gamma_t(a)=t^ka\;\forall\;t\in U(1)\big\} ,
\end{equation}
then the (vector space) direct sum $\bigoplus_{k\in\Z}\mathcal{L}_k$ is a dense *-subalgebra of $C^*(E)$ which is $\Z$-graded in the sense that $\mathcal{L}_j\cdot\mathcal{L}_k\subseteq\mathcal{L}_{j+k}$ for all $j,k\in\Z$. The involution maps each $\mathcal{L}_k$ to $\mathcal{L}_{-k}$.
In terms of generators, the grading assigns to each $S_e$ degree $+1$, to $S_e^*$ degree $-1$ and to $P_v$ degree $0$.
The C*-subalgebra $\mathcal{L}_0=:C^*(E)^{U(1)}$ of gauge-invariant elements is called the \emph{core} of $C^*(E)$.

\medskip

A graph C*-algebra $C^*(E)$ is unital if and only if $E^0$ is finite, and in this case the unit is given by the element
\begin{equation}\label{eq:unitA}
1:=\sum_{v\in E^0}P_v .
\end{equation}
If $E$ is row-finite, then $C^*(E)^{U(1)}$ is an AF-algebra, cf.~eq.~(3.9) in \cite{R05}.
If $E$ is finite and without sinks, $C^*(E)$ is \emph{strongly} $\Z$-graded (cf.~\cite{Haz13,CHR19,LO22}), meaning that
\begin{equation}\label{eq:becauseof}
\mathcal{L}_j\cdot\mathcal{L}_k=\mathcal{L}_{j+k}\qquad\forall\;j,k\in\Z.
\end{equation}
The general theory of strongly $\Z$-graded algebras tells us that each $\mathcal{L}_k$ is a self-Morita equivalence $\mathcal{L}_0$-bimodule \cite{NvO82}, which we may interpret as a the bimodule of sections of a line bundle over the quantum space described by the C*-algebra $\mathcal{L}_0$.

Note that the multiplication map gives an isomorphism $\mathcal{L}_j\otimes_{\mathcal{L}_0}\mathcal{L}_k\to \mathcal{L}_j\cdot\mathcal{L}_k$ of $\mathcal{L}_0$-bimodules, hence \eqref{eq:becauseof} can be re-interpreted as
\begin{equation}\label{eq:Picsemiring}
\mathcal{L}_j\otimes_{\mathcal{L}_0}\mathcal{L}_k\cong\mathcal{L}_{j+k}\qquad\forall\;j,k\in\Z.
\end{equation}

\medskip

Recall that an \emph{undirected walk} of length $n\geq 0$ in $E$ is a sequence
\begin{equation}\label{eq:walk}
\alpha=v_0e_1v_1e_2v_2\cdots e_nv_n
\end{equation}
of vertices $v_i\in E^0$ and edges $e_i\in E^1$ such that, for all $1\leq i\leq n$, either $s(e_i)=v_{i-1}$ and $r(e_i)=v_i$, or $s(e_i)=v_i$ and $r(e_i)=v_{i-1}$.
If, for all $1\leq i\leq n$, one has $s(e_i)=v_{i-1}$ and $r(e_i)=v_i$, then $\alpha$ is called a
\emph{directed walk} of length $n$. For a directed walk \eqref{eq:walk} we define $s(\alpha):=v_0$ and $r(\alpha):=v_n$.

A directed walk \eqref{eq:walk} of length $\geq 1$ is called \emph{closed} if $v_0=v_n$, and a closed directed walk \eqref{eq:walk} is called a \emph{directed cycle} if $v_i\neq v_j$ for all $i,j\in\{1,\ldots,n\}$ such that $i\neq j$ (any closed directed walk of length $\geq 1$ is a concatenation of directed cycle). An edge $e$ such that $s(e)=r(e)$ is called a \emph{loop}.
Note that in the graph-algebra literature directed walks are called paths (while in the standard graph theory terminology a path is a walk with no repeated vertices),
closed directed walks are called loops, and directed cycles are called simple loops (this is a terminology borrowed from theoretical physics, where in addition loops are called \emph{bubbles}). \linebreak
In this paper we will stick to the standard graph theory terminology.

\smallskip

We denote by $E^\star$ the set of all finite directed walks in $E$ and,
for $\alpha\in E^\star$ as in \eqref{eq:walk} we define
\[
S_\alpha :=P_{v_0}S_{e_1}P_{v_1}S_{e_2}P_{v_2}\cdots S_{e_n}P_{v_n} .
\]
It follows from Cuntz-Krieger relations that $S_\alpha =S_{e_1}S_{e_2}\cdots S_{e_n}$
if $n\geq 1$.

The monomials
\[
\big\{ S_\alpha S^*_\beta :\alpha,\beta\in E^\star,r(\alpha)=r(\beta)\big\}.
\]
span a dense subalgebra $L_{\C}(E)\subset C^*(E)$ called the \emph{Leavitt path algebra} of $E$ (see \cite{AAM17}). Because of \eqref{eq:CK3}, these monomials are not linearly independent.
The core $C^*(E)^{U(1)}$ is the closure of the span of the monomials $S_\alpha S^*_\beta$ with $|\alpha|=|\beta|$ (here $|\alpha|$ denotes the length of $\alpha$).

From \cite[Lemma 1.2.12]{AAM17}:
\begin{equation}\label{eq:relations}
S^*_\beta S_\alpha=
\begin{cases}
S_\gamma & \text{if } \alpha=\beta\gamma , \\
S_\gamma^* & \text{if } \beta=\alpha\gamma , \\
0 & \text{otherwise} .
\end{cases}
\end{equation}
In particular, $\{S_\alpha S^*_\alpha:\alpha\in E^\star\}$ is a set of commuting projections.
These projections generate a commutative C*-subalgebra of $C^*(E)^{U(1)}$ called the \emph{diagonal subalgebra}.

\medskip

A subset $V\subset E^0$ is called \emph{hereditary} if, for every $e\in E^1$, $s(e)\in V$ implies $r(e)\in V$. Thus, $V$ is hereditary if every directed walk that starts in $V$ stays in $V$. A subset $V\subset E^0$ is called \emph{saturated} if, for every $v\in E^0$,
$r(s^{-1}(v))\subset V$ implies $v\in V$. Thus, $V$ is saturated if every vertex emitting only edges ending in $V$ belongs to $V$.

Let $I_V$ be the closed two-sided *-ideal in $C^*(E)$ generated by $\{P_v:v\in V\}$.
If $E$ is row-finite, then $I_V$ is a gauge-invariant ideal and $C^*(E)/I_V\cong C^*(E\smallsetminus V)$, where $E\smallsetminus V$ is the subgraph of $E$ obtained by removing all vertices in $V$ and all edges with range in $V$ (cf.~Thm.~4.1 in \cite{BPRS00}).
The situations is more complicated for graphs that are not row-finite, see e.g.~Thm.~5.13 in \cite{R05}.

\medskip

We give now a slight reformulation of \cite[Thm.~2.1]{BPRS00} that is more suitable for the purposes of this work (see also \cite[Thm.~2.2]{R05}).

\begin{thm}[Gauge-invariant Uniqueness Theorem]\label{thm:gut}
Let $E$ be a row-finite graph, let $A$ be a C*-algebra with a continuous action of $U(1)$ and $\psi:C^*(E)\to A$ a $U(1)$-equivariant *-homomorphism. If $\psi(P_v)\neq 0$ for all $v\in E^0$, then $\psi$ is injective.
\end{thm}

\section{Tensor products of AF-cores}\label{sec:3}

The Cartesian product \mbox{$E\times F$} of two graphs is defined in the obvious way: the vertex set is $E^0\times F^0$, the edge set is $E^1\times F^1$, and for every $e\in E^1$ and $f\in F^1$ the arrow $(e,f)$ has source $s(e,f):=(s(e),s(f))$ and range $r(e,f):=(r(e),r(f))$ (whatever the graph is, we will always use the same letters $s$ and $r$ to denote the source and range map).

\begin{prop}\label{prop:1}
Let $E$ and $F$ be two graphs. Then, a *-homomorphism $\psi:C^*(E\times F)\to C^*(E)\otimes C^*(F)$ is defined on generators by
\begin{equation}\label{eq:psi}
\psi(P_{v,w}) :=P_v\otimes P_w , \qquad
\psi(S_{e,f}) :=S_e\otimes S_f ,
\end{equation}
for all $v\in E^0$, $w\in F^0$, $e\in E^1$, $f\in F^1$.

If $E$ and $F$ are row-finite, or are such that every cycle has an exit, then $\psi$ is
injective.
\end{prop}

\begin{proof}
We now show that the formulas above defines a (unique) *-homomorphism $\psi$ by checking that the elements $\{\psi(P_{v,w}),\psi(S_{e,f})\}$ form a Cuntz-Krieger
$E\times F$-family.
For all $(e,f)\in E^1\times F^1$, one has
\[
\psi(S_{e,f})^*\psi(S_{e,f})=
S_e^*S_e\otimes S_f^*S_f=P_{r(e)}\otimes P_{r(f)}=\psi(P_{r(e),r(f)}) ,
\]
thus proving \eqref{eq:CK1}.
Next, if $(v,w)$ is a regular vertex, that means that $s^{-1}(v,w)=s^{-1}(v)\times s^{-1}(w)$ is finite and non-empty, then both $v$ and $w$ are regular vertices, and we have
\[
\sum_{(e,f)\in s^{-1}(v,w)}\psi(S_{e,f})\psi(S_{e,f})^*=
\sum_{e\in s^{-1}(v)}S_eS_e^*\otimes \sum_{f\in s^{-1}(w)}S_fS_f^*=P_{v}\otimes P_{w}=\psi(P_{v,w}) .
\]
This proves \eqref{eq:CK3}.
Finally, for all $e\in E^1$ and $f\in F^1$, one has $S_eS_e^*\leq P_{s(e)}$, $S_fS_f^*\leq P_{s(f)}$. Since the tensor product of two positive operators is positive, for all $(e,f)\in E^1\times F^1$ one has
\[
P_{s(e)}\otimes P_{s(f)}-S_eS_e^*\otimes S_fS_f^*
=(P_{s(e)}-S_eS_e^*)\otimes P_{s(f)}
+S_eS_e^*\otimes (P_{s(f)}-S_fS_f^*) \geq 0 ,
\]
hence
\[
\psi(S_{e,f})\psi(S_{e,f}^*)=S_eS_e^*\otimes S_fS_f^*\leq 
P_{s(e)}\otimes P_{s(f)}=\psi(P_{s(e,f)}) .
\]
This proves \eqref{eq:CK2}.

Since all the Cuntz-Krieger relations are satisfied, the *-homomorphism $\psi$ exists (and is unique). The map $\psi$ is $U(1)$-equivariant if on the codomain we consider the $U(1)$ gauge action only on the first factor (or only on the second factor):
\begin{equation}\label{eq:gauge}
\psi\circ\gamma_z=(\gamma_z\otimes\id)\circ\psi=(\id\otimes\gamma_z)\circ\psi , \qquad
\forall\;z\in U(1).
\end{equation}
If $E$ and $F$ are row-finite (that is: $E\times F$ is row-finite), since $\psi(P_{v,w})\neq 0$ for all $v,w$, from Thm.~\ref{thm:gut} it follows that $\psi$ is injective.
If $E$ and $F$ are not row-finite, but they are such that every cycle has an exit, then it follows from Thm.~2 in \cite{FLR00} that $\psi$ is injective.
\end{proof}

Prop.~\ref{prop:1} already appeared in \cite{Cha10} under the hypothesis that the graphs are (1) row-finite, (2) such that every cycle has an exit, and (3) with no sinks. But in fact the hypothesis (3) is not necessary, and only one between (1) and (2) is needed for injectivity. A similar result also appeared in \cite{IPL16} in the context of ultragraph C*-algebras.

\begin{lemma}\label{lemma:plusone}
Let $E$ be a row-finite graph with no sinks, and let $k\geq 0$. Then, every monomial $S_\alpha S_\beta^*$ can be written as a sum of elements $S_{\alpha'}S_{\beta'}^*$ with
$|\alpha'|=|\alpha|+k$ and $|\beta'|=|\beta|+k$.
\end{lemma}

\begin{proof}
It is enough to prove the Lemma when $k=1$. Let $v:=r(\alpha)$ and notice that, by assumption, $v\in E_{\mathrm{reg}}^0$. 
From \eqref{eq:relations}, $S_\alpha P_v S_\beta^*=S_\alpha S_\beta^*$. Using \eqref{eq:CK3}:
\[
S_\alpha S_\beta^*=S_\alpha P_v S_\beta^*=\sum_{e\in s^{-1}(v)}S_\alpha S_eS_e^*S_\beta^* ,
\]
which is what we wanted to prove.
\end{proof}

\begin{prop}\label{prop:2}
Let $E$ and $F$ be row-finite graphs with no sinks. Then, map $\psi$ in Prop.~\ref{prop:1} restricts to an isomorphism
\[
C^*(E)^{U(1)}\otimes C^*(F)^{U(1)} \cong C^*(E\times F)^{U(1)} .
\]
\end{prop}

\begin{proof}
It follows from \eqref{eq:gauge} that $\psi$ maps $C^*(E\times F)^{U(1)}$ into the subalgebra of the codomain consisting of $U(1)\times U(1)$ invariant elements,
\begin{equation}\label{eq:induced}
\psi:C^*(E\times F)^{U(1)}\to C^*(E)^{U(1)}\otimes C^*(F)^{U(1)} .
\end{equation}
It follows from Prop.~\ref{prop:1} that the map \eqref{eq:induced} is injective. We must show that it is also surjective.

The codomain of \eqref{eq:induced} is spanned by elements of the form
\[
a:=S_\alpha S_\beta^*\otimes S_\mu S_\nu^*
\]
with $|\alpha|=|\beta|=:n$ and $|\mu|=|\nu|=:k$. We must show that $a$ is in the image of $\psi$.

Without loss of generality we can assume that $k=n\geq 1$ (we can use Lemma \ref{lemma:plusone} to increase the lengths of either $\alpha$ and $\beta$, $\mu$ and $\nu$, or all of them).
For $k=n$, we have
\[
a=
(S_{\alpha_1}\otimes S_{\mu_1})
(S_{\alpha_2}\otimes S_{\mu_2})
\cdots
(S_{\alpha_n}\otimes S_{\mu_n})
(S_{\beta_n}^*\otimes S_{\nu_n}^*)
\cdots
(S_{\beta_2}^*\otimes S_{\nu_2}^*)
(S_{\beta_1}^*\otimes S_{\nu_1}^*)
\]
where $\alpha_1,\alpha_2,\ldots\in E^1$, $\alpha=\alpha_1\alpha_2\ldots\alpha_n$, etc.

Each parenthesis in the above equality is an element in the image of $\psi$.
\end{proof}

As a special case of Prop.~\ref{prop:1} and Prop.~\ref{prop:2}, we find a well-known property of Cuntz algebras.

\begin{ex}[Cuntz algebra \protect{\cite{Cun81}}]\label{ex:Cuntz}
Let $n\geq 1$ and $B_n$ be the graph of the Cuntz algebra $\mathcal{O}_n$,
\pagebreak
given by one vertex with $n$ loops (the \emph{bouquet} graph):
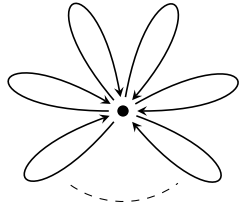
\begin{center}
\begin{tikzpicture}

\clip (-1.6,-1.3) rectangle (1.6,1.55);

\node[main node] (1) {};

\filldraw (1) circle (0.06);

\path[freccia] (1) edge[ciclo,out=30, in=0, shorten <=3pt, shorten >=3pt] (1);
\path[freccia] (1) edge[ciclo,out=80, in=50, shorten <=3pt, shorten >=3pt] (1);
\path[freccia] (1) edge[ciclo,out=130, in=100, shorten <=3pt, shorten >=3pt] (1);
\path[freccia] (1) edge[ciclo,out=180, in=150, shorten <=3pt, shorten >=3pt] (1);
\path[freccia] (1) edge[ciclo,out=230, in=200, shorten <=3pt, shorten >=3pt] (1);
\path[freccia] (1) edge[ciclo,out=340, in=310, shorten <=3pt, shorten >=3pt] (1);

\draw[dashed] (235:1.2) arc (235:307:1.2);

\end{tikzpicture}
\end{center}
Note that $B_m\times B_n\cong B_{m\cdot n}$. By Prop.~\ref{prop:1} we have an injective *-homomorphism
\(
\mathcal{O}_{m\cdot n}\to\mathcal{O}_m\otimes\mathcal{O}_n
\).
By Prop.~\ref{prop:2}, recalling that the AF-core of $\mathcal{O}_n$ is $UHF(n^\infty)$,
we get the isomorphism
\[
UHF( (m\cdot n)^\infty) \cong
UHF(m^\infty)\otimes UHF(n^\infty) .
\]
\end{ex}

\section{A contravariant functoriality}\label{sec:4}

The purpose of this section is to review the notion of ``$\cup$-admissible'' morphisms of \cite{HT24}. These are special cases of more general ``admissible'' morphisms forming a category that is denoted CRTBPOG in \cite{HT24}.
However, since we don't need the more general notion, in an attempt to make this paper as much self-contained as possible
we streamline the definition and proofs of the properties of $\cup$-admissible morphisms, which in this paper will be simply called admissible.

\smallskip

Firstly, a morphism $\phi:E\to F$ of graphs is a pair of maps $\phi_0:E^0\to F^0$ and $\phi_1:E^1\to F^1$ satisfying $\phi_0\circ s=s\circ\phi_1$ and 
$\phi_0\circ r=r\circ\phi_1$.
If $\phi$ is injective (i.e.~both $\phi_0$ and $\phi_1$ are injective maps), we can identify $E$ with a subgraph of $F$. In the following (and in fact also in Example \ref{ex:Cuntz}), we tacitly use the fact that isomorphic graphs have isomorphic graph C*-algebras.

An \emph{induced subgraph} is a subgraph $E\subset F$ such that all edges of $F$ with source and range in $E$ belong to $E$, in formulas:
\[
E^1=\big\{f\in F^1:s(f)\in E^0,r(f)\in E^0\big\}
\]
(we say that $E$ is \emph{induced} by its vertex set). Note that $E^1$ is always contained in the set on the right hand side of the above equality, and only the opposite inclusion is non-trivial.

\begin{prop}\label{lemma:41}
If $E$ is an induced subgraph of a row-finite graph $F$ such that $F^0\smallsetminus E^0$ is hereditary and saturated, then there is a $U(1)$-equivariant surjective *-homomorphism $C^*(F)\to C^*(E)$.
\end{prop}

\begin{proof}
Since $V:=F^0\smallsetminus E^0$ is hereditary and saturated, we have a $U(1)$-equivariant surjective *-homomorphism $C^*(F)\to C^*(F\smallsetminus V)$. We now show that $F\smallsetminus V=E$.
Clearly $(F\smallsetminus V)^0=E^0$. Recall that $(F\smallsetminus V)^1$ is the set of all edges whose range is not in $V$. But $V$ is hereditary: if $r(f)\notin V$ then also $s(f)\notin V$. Thus, $(F\smallsetminus V)^1$ is the set of all edges $f\in F^1$ with both source and range in $E^0$, which means $f\in E^1$ since $E$ is an induced subgraph. Thus, $(F\smallsetminus V)^1=E^1$.
\end{proof}

For an induced subgraph $E\subset F$, the condition that $F^0\smallsetminus E^0$ is hereditary is equivalent to the condition that, for all $f\in F^1$:
\begin{equation}\label{eq:herbis}
r(f)\in E^0 \Longrightarrow f\in E^1 .
\end{equation}
The condition that $F^0\smallsetminus E^0$ is saturated is equivalent to the condition that
\begin{equation}\label{eq:satbis}
\text{if $v\in E^0$ is not a sink in $F$, then $\exists\;e\in E^1$ such that $s(e)=v$}
\end{equation}
(in other words, if $v\in E^0$ is not a sink in $F$, it is also not a sink in $E$).
Note that \eqref{eq:herbis} \emph{implies} that $E$ is an induced subgraph.
In view of Prop.~\ref{lemma:41}, we then give the following definition.

\begin{df}[\protect{\cite{CHT21,HRT20,HT24}}]\label{def:adm}
A morphism $\phi:G\to F$ of directed graphs is called
\emph{admissible} if:
\begin{enumerate}
\item $\phi$ is injective,
\item $E:=\phi(G)$ satisfies \eqref{eq:herbis} and \eqref{eq:satbis}.
\end{enumerate}
\end{df}

This is a exactly Definition 3.1 in \cite{HT24}, and was also considered in \cite{HRT20} and \cite{CHT21}.
With the above definition, we can rephrase Prop.~\ref{lemma:41} by saying that:

\begin{cor}\label{cor:1}
Every admissible morphism $\phi:G\to F$ between row-finite graphs induces a $U(1)$-equivariant surjective *-homomorphism $C^*(\phi):C^*(F)\to C^*(G)$. This is explicitly given on generators,
with a slight abuse of notations, by the map
\begin{align*}
S_f &\mapsto S_e \quad\text{if }f=\phi(e), &
S_f &\mapsto 0 \quad\text{otherwise}, \\
P_w &\mapsto P_v \quad\text{if }w=\phi(v), &
P_w &\mapsto 0 \quad\text{otherwise},
\end{align*}
for all $w\in F^0$ and $f\in F^1$.
\end{cor}

For the sake of completeness, let us show that we get a category and a functor.
This category is the subcategory of CRTBPOG formed by row-finite graphs and injective morphisms (cf. Prop.~3.4 \cite{HT24}), and the functor is a restriction of the one in \cite[Cor.~4.8]{HT24}.

\begin{prop}
Row-finite graphs with admissible morphisms form a category.
We denote this category by $\mathsf{AdGrph}$.
\end{prop}

\begin{proof}
We must show that $\mathsf{AdGrph}$ is a subcategory of the category of (row-finite) directed graphs.
Clearly, for every row-finite graph $E$, the identity morphism is an admissible morphism.
Consider two admissible morphisms
\[
G\xrightarrow{\quad\phi\quad}F\xrightarrow{\quad\psi\quad}H .
\]
Since $\phi$ and $\psi$ are injective, their composition is injective.
We must show that $\psi(\phi(G))$ satisfies \eqref{eq:herbis} and \eqref{eq:satbis}.

\smallskip

We start with \eqref{eq:herbis}.
Let $h\in H^1$ be such that $r(h)\in\psi_0(\phi_0(G^0))$. Then $r(h)\in\psi_0(F^0)$ and, by assumption, this means that $h\in\psi_1(F^1)$. Thus, $h=\psi_1(f)$ for some $f\in F^1$.
Since $r(h)=\psi_0(r(f))$ belongs to $\psi_0(\phi_0(G^0))$.
Since $\psi_0$ is injective, this implies that $r(f)$ belongs to $\phi_0(G^0)$. By assumption, this implies $f\in\phi_1(G^1)$.
Thus, $h=\psi_1(f)\in\psi_1(\phi_1(G^1))$, which is what we wanted to prove.

\smallskip

Next, we pass to \eqref{eq:satbis}. Let $v\in\psi_0(\phi_0(G^0))$ be not a sink.
From the property \eqref{eq:satbis} for $\psi(F)$ we deduce that there exists $h=\psi_1(f)\in\phi_1(F^1)$ with $s(h)=v$. Since $s(h)=\psi_0(s(f))\in\psi_0(\phi_0(G^0))$ and
$\psi_0$ is injective, we deduce that $s(f)\in\phi_0(G^0)$. Since $s(f)$ is obviously not a sink, from the property \eqref{eq:satbis} for $\phi(G)$ we deduce that there exists $f'\in\phi_1(G^1)$ with $s(f')=s(f)$. Let $h'\in\psi_1(\phi_1(G^1))$ be the edge $h':=\psi_1(f')$. Then
\[
s(h')=\psi_0(s(f'))=\psi_0(s(f))=s(h)=v ,
\]
thus concluding the proof.
\end{proof}

\begin{prop}
The associations
\[
E\mapsto C^*(E)\qquad\text{and}\qquad\phi\mapsto C^*(\phi)
\]
give a contravariant functor $C^*$ from the category $\mathsf{AdGrph}$ to the category of $U(1)$-C*-algebras and surjective *-homomorphisms.
\end{prop}

\begin{proof}
It follows from Cor.~\ref{cor:1} that $C^*(\phi)$ is a morphism in the target category.
Clearly the identity map $E\to E$ is an admissible morphism, and in this case $C^*(\phi)$ is the identity
(the quotient by the zero  ideal).
Let
\[
E\xrightarrow{\quad\phi\quad}F\xrightarrow{\quad\psi\quad}G .
\]
be two admissible morphisms. We now show that $C^*(\psi\circ\phi)=C^*(\phi)\circ C^*(\psi)$. More precisely, call $I_\phi$, $I_\psi$ and $I_{\psi\circ\phi}$ the kernels of $C^*(\phi)$, $C^*(\psi)$ and $C^*(\psi\circ\phi)$, respectively. We now show that \begin{equation}\label{eq:inclusion}
I_\psi\subset I_{\psi\circ\phi}
\end{equation}
and that
\begin{equation}\label{eq:quotient}
I_{\psi\circ\phi}/I_\psi=C^*(\psi)(I_{\psi\circ\phi})=I_\phi .
\end{equation}
Denote by $\{S_e,P_v\}$ the Cuntz-Krieger $G$-family in $C^*(G)$ and by $\{T_f,Q_w\}$ the Cuntz-Krieger $F$-family in $C^*(F)$.
Since $\psi_0$ is injective, $G^0$ is the disjoint union of three subsets
\begingroup
\setlength{\arraycolsep}{1em}
\[
\begin{array}{ccc}
G^0\smallsetminus\psi_0(F^0) , &
\psi_0\big(F^0\smallsetminus\phi_0(E^0)\big) , &
\psi_0\big(\phi_0(E^0)\big) . \\[2pt]
\text{(I)} &
\text{(II)} &
\text{(III)}
\end{array}
\]
\endgroup

The ideal $I_\psi$ is generated by the vertex projections $P_v$ with $v$ in the set (I).
The ideal $I_{\psi\circ\phi}$ is generated by the vertex projections $P_v$ with $v\in G^0\smallsetminus\psi_0\big(\phi_0(E^0)\big)$, that means that $v$ is either in the set (I) or in the set (II). From this observation it follows the inclusion \eqref{eq:inclusion}.

The ideal $I_\phi$ is generated by the vertex projections $Q_w$ with $w\in F^0\smallsetminus\phi_0(E^0)$. Since $\psi_0$ is injective, this is equivalent to the condition that $\psi_0(w)$ belongs to the set (II).

For $v$ in the set (I), $C^*(\psi)(P_v)=0$.
For $v=\psi_0(w)$ in the set (II),
\[
C^*(\psi)(P_v)=Q_w \in I_\phi .
\]
Thus, $C^*(\psi)(I_{\psi\circ\phi})\subset I_\phi$. Since every generator is in the image, this inclusion is actually an equality, which proves \eqref{eq:quotient}.
\end{proof}

\begin{ex}[Trimmable graphs  \protect{\cite{ADHT22}}]\label{ex:trim}
Let $F$ be a row-finite graph, $v_0\in F^0$, and let $G\subset F$ be the subgraph given by $G^0:=F^0\smallsetminus\{v_0\}$ and $G^1:=F^1\smallsetminus r^{-1}(v_0)$.
By construction, $G$ is an induced subgraph.
If $F$ is $v_0$-trimmable (see Definition 2.1 in \cite{ADHT22}),
then $\{v_0\}$ is hereditary and saturated in $F$, and
the inclusion of $G$ into $F$ is an admissible morphism.
\end{ex}

\begin{ex}[Quantum spheres \protect{\cite{HS02}}]\label{ex:qsphere}
A special case of Example \ref{ex:trim} is given by the quantum spheres $\mathbb{S}^{2n-1}_q$ of Vaskman-Soibelman \cite{VS91}. 
For $n=1$, $\mathbb{S}^1_q=\mathbb{S}^1$ is the unit circle.
For $n=2$, $\mathbb{S}^3_q$ is the quantum space underlying the quantum $SU(2)$ group.
For $n\geq 1$, $C(\mathbb{S}^{2n-1}_q)\cong C^*(\Sigma_{2n-1})$,
where $\Sigma_{2n-1}$ is the graph in the picture below:\smallskip
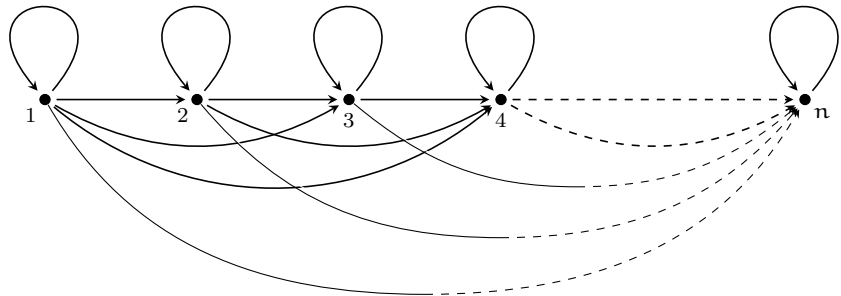
\begin{center}
\begin{tikzpicture}[inner sep=3pt,font=\footnotesize]

\clip (-0.6,-2.7) rectangle (10.6,1.4);

\node[main node] (1) {};
\node (2) [main node,right of=1] {};
\node (3) [main node,right of=2] {};
\node (4) [main node,right of=3] {};
\node (5) [right of=4] {};
\node (6) [main node,right of=5] {};

\filldraw (1) circle (0.06) node[below left] {$1$};
\filldraw (2) circle (0.06) node[below left] {$2$};
\filldraw (3) circle (0.06) node[below=2pt] {$3$};
\filldraw (4) circle (0.06) node[below=2pt] {$4$};
\filldraw (6) circle (0.06) node[below right] {$n$};

\path[freccia] (1) edge[ciclo] (1);
\path[freccia] (2) edge[ciclo] (2);
\path[freccia] (3) edge[ciclo] (3);
\path[freccia] (4) edge[ciclo] (4);
\path[freccia] (6) edge[ciclo] (6);

\path[freccia]
	(1) edge (2)
	(2) edge (3)
	(3) edge (4)
	(4) edge[dashed]  (6)
	(1) edge[bend right] (3)
	(1) edge[bend right=40] (4)
	(2) edge[bend right] (4);

\path[white]
	(1) edge[bend right=60] coordinate (7) (6)
	(2) edge[bend right=50] coordinate (8) (6)
	(3) edge[bend right=40] coordinate (9) (6);

\path
	(1) edge[out=-60,in=180] (7)
	(2) edge[out=-50,in=180] (8)
	(3) edge[out=-40,in=180] (9);

\path[->,dashed,shorten >=2pt]
	(7) edge[out=0,in=240] (6)
	(8) edge[out=0,in=230] (6)
	(9) edge[out=0,in=220] (6);

\path[freccia,dashed] (4) edge[bend right,dashed] (6);
\end{tikzpicture}\smallskip
\end{center}
This graph is denoted by $L_{2n-1}$ in \cite{HS02}, but we avoid using the letter $L$ since it is already used for line bundles. The graph $\Sigma_{2n-1}$ has vertex set $\Sigma_{2n-1}^0:=\{1,\ldots,n\}$ and one edge from the vertex $i$ to the vertex $j$ for all $1\leq i\leq j\leq n$.

For $n\geq 2$, the embedding $\Sigma_{2n-3}\to\Sigma_{2n-1}$ as the subgraph consisting of the vertices $i$ and edges $(i,j)$
of $\Sigma_{2n-1}$ with $i,j\leq n-1$ is an admissible morphism.
\end{ex}

\begin{ex}[Multichamber quantum 2-spheres \protect{\cite{HT24}}]\label{ex:c2k}
Let $k\geq 0$ and let $C^2_{k-1}$ be the
\pagebreak
graph with $k+1$ vertices, a loop at a vertex $v_0$, and an edge from $v_0$ to $w$ for all $w\neq v_0$:
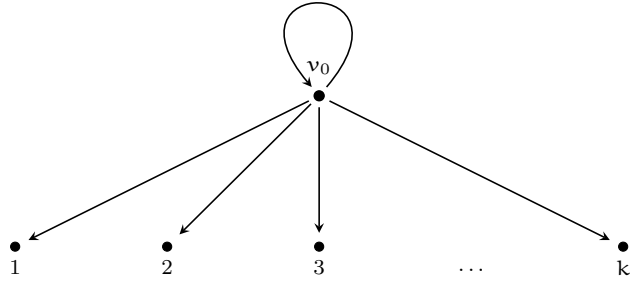
\begin{center}
\begin{tikzpicture}[inner sep=3pt,font=\footnotesize]

\node[main node] (1) {};
\node (4) [below of=1] {};
\node (3) [left of=4] {};
\node (2) [left of=3] {};
\node (5) [right of=4] {};
\node (6) [right of=5] {};

\draw (5) node[below=5pt] {$\ldots$};
\filldraw (1) circle (0.06) node[above=5pt] {$v_0$};
\filldraw (2) circle (0.06) node[below=2pt] {$1$};
\filldraw (3) circle (0.06) node[below=2pt] {$2$};
\filldraw (4) circle (0.06) node[below=2pt] {$3$};
\filldraw (6) circle (0.06) node[below=2pt] {$k$};

\path[freccia] (1) edge[ciclo] (1) (1) edge (2) (1) edge (3) (1) edge (4) (1) edge (6);

\end{tikzpicture}\smallskip
\end{center}
For $k\geq 2$, $C^*(C^2_{k-1})$ is the C*-algebra of what is called a \emph{multichamber quantum 2-sphere} in \cite{HT24}.
If $k=2$, this is the equatorial Podle\'s sphere \cite{HS02}, if $k=1$ we get the Toeplitz algebra, and if $k=0$ we get $C(\mathbb{S}^1)$.
The obvious embedding
\[
C^2_{k-1}\to C^2_k
\]
is an admissible morphism for all $k\geq 0$.

Observe that the induced *-homomorphism $C^*(C^2_k)\to C^*(C^2_{k-1})$ goes in the opposite direction w.r.t.~the one in \cite[Eq.~(7.9)]{HS02}. For $k=0$, this *-homomorphism is the well known symbol map.
\end{ex}

\begin{ex}[Quantum lens spaces \protect{\cite{HS08}}]\label{ex:qlens}
For $k\geq 0$, let $L_3(k)$ be the graph obtained from the graph $C^2_{k-1}$ in Example
\ref{ex:c2k} by adding one loop at every vertex different from $v_0$:
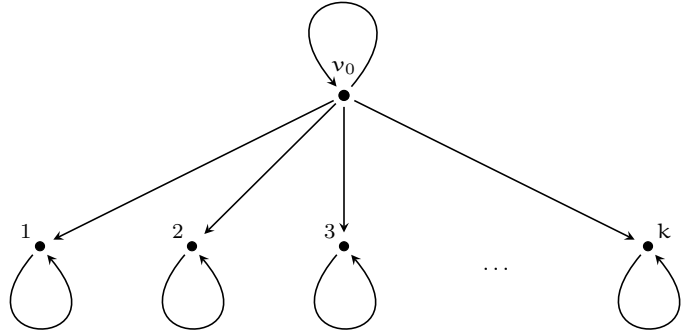
\begin{center}
\begin{tikzpicture}[inner sep=3pt,font=\footnotesize]

\node[main node] (1) {};
\node (4) [below of=1] {};
\node (3) [left of=4] {};
\node (2) [left of=3] {};
\node (5) [right of=4] {};
\node (6) [right of=5] {};

\draw (5) node[below=5pt] {$\ldots$};
\filldraw (1) circle (0.06) node[above=5pt] {$v_0$};
\filldraw (2) circle (0.06) node[above left] {$1$};
\filldraw (3) circle (0.06) node[above left] {$2$};
\filldraw (4) circle (0.06) node[above left] {$3$};
\filldraw (6) circle (0.06) node[above right] {$k$};

\path[freccia] (1) edge[ciclo] (1) (1) edge (2) (1) edge (3) (1) edge (4) (1) edge (6);
\path[out=230, in=310, loop, distance=1.7cm, ->,semithick] (2) edge (2) (3) edge (3) (4) edge (4) (6) edge (6);

\end{tikzpicture}
\end{center}
For $k\geq 2$ we get the C*-algebra of the quantum Lens space $L^3_q(k;1,k)$ \cite{HS08}. For $k=0$ we get again $C(\mathbb{S}^1)$ and for $k=1$ we get $SU_q(2)$ \cite{HS02}. The obvious embedding
\[
L_3(k)\to L_3(k+1)
\]
is an admissible morphism for all $k\geq 0$.
\end{ex}

\section{Admissible morphisms $E\to E\times E$}\label{sec:5}

In this section we study two types of morphisms $E\to E\times E$: vertical embeddings (Sect.~\ref{sec:51}) and the diagonal embedding (Sect.~\ref{sec:52}).
The former embedding is admissible for a class of graphs which includes many interesting examples, such as Vaksman-Soibelman quantum spheres. As a result of this construction, one gets a ring structure on the K-theory of a quantum projective space, although not a very interesting one. The latter embedding is admissible for graphs such that every vertex has at most one entry. For these graphs, the AF core is commutative.

\subsection{Vertical embeddings}\label{sec:51}

\begin{prop}\label{prop:const}
Let $G$ and $E$ be directed graphs, and assume that there is a loop $\ell_0\in G^1$. Call $w_0:=s(\ell_0)$. Then, a morphism $\phi=(\phi_0,\phi_1):E\to G\times E$ given by
\[
\phi_0(v):=(w_0,v) , \qquad \phi_1(e):=(\ell_0,e) ,
\]
for all $v\in E^0$ and $e\in E^1$.
The morphism $\phi$ is admissible if and only if $r^{-1}(w_0)=\{\ell_0\}$.
\end{prop}

\begin{proof}
Since $s(\phi_1(e))=(w_0,s(e))=\phi_0(s(e))$ and $r(\phi_1(e))=(w_0,r(e))=\phi_1(r(e))$,
the pair $(\phi_0,\phi_1)$ is indeed a morphism of graphs.

If $v\in E^0$ is a sink in $E$, then $\phi_0(v)$ is a sink in $G\times E$.
If $v$ is not a sink, given any edge $e\in E^1$ with $s(e)=v$, the edge $(\ell_0,e)\in G^1\times E^1$ has source $\phi_0(v)$ and range $(w_0,r(e))=\phi_0(r(e))$ in the image of $\phi_0$. Thus, \eqref{eq:satbis} is satisfied.

Suppose $(g,e)\in G^1\times E^1$ has range in $\phi_0(E^0)$, that is $r(g,e)=(w_0,v)$ for some $v\in E^0$. If there is an edge $g\neq \ell_0$ with range $w_0$, then $(g,e)\notin\phi_0(E^0)$ and \eqref{eq:herbis} is not satisfied. Conversely, if the only edge with range $w_0$ is $\ell_0$, then $(g,e)=(\ell_0,e)$ for every $(g,e)$ with range in $\phi_0(E^0)$, and so $(g,e)=\phi(e)\in\phi_1(E^1)$ and \eqref{eq:herbis} is satisfied.
\end{proof}

The graph of a quantum sphere (Examples \ref{ex:qsphere}),
the graph of a multichamber quantum 2-sphere (Esample \ref{ex:c2k})
and the graph of a quantum lens space (Example \ref{ex:qlens}) all satisfy the hypothesis of the Prop.~\ref{prop:const}.
Observe that the graphs of a quantum sphere and of a quantum lens space have no sinks.

If $G$ and $E$ are row-finite with no sinks, combining Prop.~\ref{prop:const} with Prop.~\ref{prop:2} one gets a *-homomorphism
\[
C^*(G)^{U(1)}\otimes C^*(E)^{U(1)}\longrightarrow C^*(E)^{U(1)}
\]
and then a morphism of abelian groups
\[
K_*\big(C^*(G)^{U(1)}\big)\otimes K_*\big(C^*(E)^{U(1)}\big)\longrightarrow K_*\big(C^*(E)^{U(1)}\big) .
\]
For $G=E$ one gets a ring structure on $K_*\big(C^*(E)^{U(1)}\big)$, although not a very interesting one.

In the notations of Prop.~\ref{prop:const}, the complement $V$ of $\{w_0\}$ in $G^0$ is an hereditary and saturated subset,
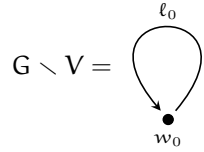
\[
G\smallsetminus V=\hspace*{-7mm}
\begin{tikzpicture}[font=\scriptsize,baseline=(current bounding box.center)]

\node[main node] (1) {};
\filldraw (1) circle (0.06) node[below=2pt] {$w_0$};
\path[freccia] (1) edge[ciclo] node[above] {$\ell_0$} (1);

\end{tikzpicture}
\]
and the above *-homomorphism factors as follows
\[
C^*(G)^{U(1)}\otimes C^*(E)^{U(1)}\longrightarrow 
\underbrace{C(\mathbb{S}^1)^{U(1)}}_{=\C}\otimes C^*(E)^{U(1)} \longrightarrow 
C^*(E)^{U(1)} .
\]
The map in K-theory induced by the second arrow is simply the multiplication by $\Z$.

\smallskip

On the other hand, in the case of quantum spheres the one just discussed is the only admissible embedding $E\to E\times E$. Let us illustrate this point when $E=\Sigma_3$ is the graph of $SU_q(2)$:
\begin{equation}\label{eq:SUq2}
\begin{tikzpicture}[nodogrigio/.style={circle,inner sep=2pt,fill=gray!20},font=\scriptsize,baseline=(current bounding box.center)]

\node[nodogrigio] (1) {$1$};
\node (2) [nodogrigio,right of=1] {$2$};

\path[freccia] (1) edge[ciclo] (1);
\path[freccia] (2) edge[ciclo] (2);
\path[freccia] (1) edge (2);

\end{tikzpicture}
\end{equation}
The product $E\times E$ is the graph
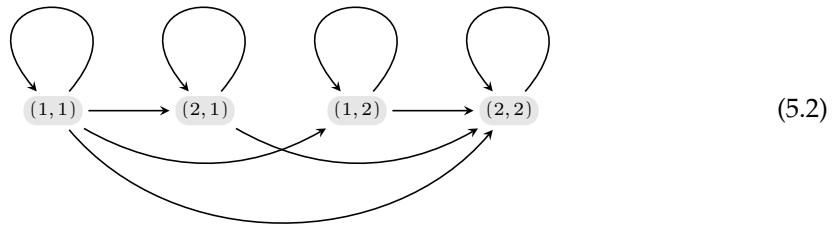
\begin{equation}\label{eq:graphSuq2}
\begin{tikzpicture}[nodogrigio/.style={inner sep=2pt,fill=gray!20,rectangle,rounded corners=5pt},font=\scriptsize,baseline=(current bounding box.center)]

\node[nodogrigio] (1) {$(1,1)$};
\node (2) [nodogrigio,right of=1] {$(2,1)$};
\node (3) [nodogrigio,right of=2] {$(1,2)$};
\node (4) [nodogrigio,right of=3] {$(2,2)$};

\path[freccia] (1) edge[ciclo] (1);
\path[freccia] (2) edge[ciclo] (2);
\path[freccia] (3) edge[ciclo] (3);
\path[freccia] (4) edge[ciclo] (4);
\path[freccia]
	(1) edge (2)
	(1) edge[bend right=30] (3)
	(1) edge[bend right=50] (4)
	(2) edge[bend right=30] (4)
	(3) edge (4);

\end{tikzpicture}\vspace{-3pt}
\end{equation}
Note that this is obtained from the graph $\Sigma_7$ of $\mathbb{S}^7_q$ by removing one arrow.

An embedding $E\to E\times E$ corresponds to a choice of two vertices in the graph \eqref{eq:graphSuq2}, which identify a unique induced subgraph. One can study all possible cases (six) and find that there are only two choices of vertices giving an admissible embedding. One can choose the vertices $(1,1)$ and $(1,2)$, and the corresponding embedding is the vertical embedding with $w_0=1$ and $\ell_0$ the loop at $1$. Otherwise, one can choose the vertices $(1,1)$ and $(2,1)$, and the corresponding embedding is the map $v\mapsto (v,w_0)$, $e\mapsto (e,\ell_0)$, with $w_0=1$ and $\ell_0$ the loop at $1$ (the horizontal embedding symmetric to the vertical embedding above). The diagonal embedding, discussed in the next section, corresponds to the choice of vertices $(1,1)$ and $(2,2)$, and one can check that it's not admissible (the remaining two vertices form a subset of $E^0$ that is not hereditary).

\subsection{The diagonal embedding}\label{sec:52}

The diagonal embedding $E\to E\times E$ of a graph $E$ is the injective morphism given by $v\mapsto (v,v)$ and $e\mapsto (e,e)$, for all $v\in E^0$ and $e\in E^1$.

If $S$ is any set, $D\subseteq S$ a subset, and $f:D\to S$ a function, we can define a graph $E$ by $E^0:=S$ and $E^1:=\{(x,y)\in D\times S:y=f(x)\}$. In this graph, every vertex emits at most one edge.
A directed graph where each vertex emits at most one edge is called a \emph{functional graph}. Every functional graph is associated to a partially defined function through the construction above.

If $E$ is any directed graph, the \emph{transposed} of $E$ is the graph obtained from $E$ by reversing the direction of all the arrows (so, the source map of $E$ is the range map of the transposed graph, and vice versa).

\begin{prop}\label{prop:diagem}
The diagonal embedding $E\to E\times E$ is admissible if and only if the transposed of $E$ is a functional graph.
\end{prop}

\begin{proof}
Given $e,f\in E$, $r(e,f)$ belongs to the diagonal if{}f $r(e)=r(f)$.
The condition \eqref{eq:herbis} translates into the condition that $r(e)=r(f)$ implies $e=f$, that is: every vertex receives at most one edge (thus, in the transposed of $E$ every vertex emits at most one edge). The condition \eqref{eq:satbis} is automatically satisfied: if $(v,v)\in E^0\times E^0$ is not a sink, $v\in E^0$ is not a sink, which means that there exists $e\in E^1$ with $s(e)=v$, and so $s(e,e)=(v,v)$.
\end{proof}

The graph of a multichamber quantum 2-sphere in Example \ref{ex:c2k} is the transposed of a functional graph.
%However, to get a ring structure on $K(C^*(C^2_{k-1})^{U(1)})$ the graph must have no sink, which happens only in the trivial (commutative) case $k=0$.
%The graph of an odd-dimensional quantum sphere $\mathbb{S}^{2n-1}_q$ in Example \ref{ex:qsphere} and the graph of a quantum lens space $L_q^3(k;1,k)$ in Example \ref{ex:qlens} are not the transposed of a functional graph, except in the trivial (commutative) case $n=1$ and $k=0$.

\begin{ex}[Matrix algebras]
Consider the two graphs
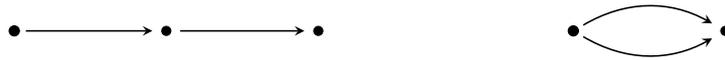
\begin{center}
\begin{tikzpicture}[inner sep=3pt,baseline=(current bounding box.center)]

\node[main node] (1) {};
\node (2) [right of=1] {};
\node (3) [right of=2] {};

\filldraw (1) circle (0.06);
\filldraw (2) circle (0.06);
\filldraw (3) circle (0.06);

\path[freccia] (1) edge (2) (2) edge (3);

\end{tikzpicture}
\hspace{3cm}
\begin{tikzpicture}[inner sep=3pt,baseline=(current bounding box.center)]

\node[main node] (1) {};
\node (2) [right of=1] {};

\filldraw (1) circle (0.06);
\filldraw (2) circle (0.06);

\path[freccia] (1) edge[bend right] (2) (1) edge[bend left] (2);

\end{tikzpicture}
\end{center}
Each one is its own transposed.
The first graph is a functional graph while the second is not, even if the graph C*-algebra is, in both cases, isomorphic to $M_3(\C)$.
\end{ex}

Recall that directed graph $G$ is called \emph{connected} if for every $v,w\in G^0$ with $v\neq w$ there is an undirected walk joining $v$ and $w$. (According to this definition, all the examples of graphs that we saw in this paper are connected.)

\begin{lemma}\label{lemma:arbo}
Let $E$ be a connected functional graph or the transposed of a connected functional graph. Then, $E$ has at most one directed cycle.
\end{lemma}

\begin{proof}
Since transposition transforms directed cycles into directed cycles, it is enough to prove the statement when $E$ is the transposed of a functional graph, i.e.~every vertex receives at most one edge.
Let $c_1:=v_0e_1v_1e_2\ldots e_kv_k$ and $c_2:=w_0f_1w_1f_2\ldots f_nw_n$ be two directed cycles.
Since $E$ is connected, either $c_1$ and $c_2$ have a common vertex, or there is an undirected walk $\alpha=u_0g_0u_1g_1u_2\ldots g_lu_l$ of length $l\geq 1$
joining a vertex of $c_1$ and a vertex of $c_2$. Assume that we are in the second case.
Up to a relabelling, we can assume that $u_0=v_0$ and $u_l=w_0$.

If $u_0=r(g_0)$, then $g_0$ and $e_k$ have the same range and, by assumption, this mean that $g_0=e_k$ and the vertex $u_1=s(g_0)=s(e_k)=v_{k-1}$ belongs to $c_1$.
If $l=1$ this means that $c_1$ and $c_2$ have a common vertex ($u_1=w_0$).
If $l\geq 2$, this mean that there is an undirected walk $u_1g_1u_2\ldots g_lu_l$ of length $l-1$ joining a vertex of $c_1$ and a vertex of $c_2$.

If $u_0=s(g_0)$, since every vertex receives at most one edge, $v_1=s(e_2)$, $v_2=s(e_3)$, etc, and the undirected path is directed. From $r(g_l)=u_l=w_0=r(f_n)$
we deduce that $g_l=f_n$ and $u_{l-1}=s(g_l)=s(f_n)=w_{n-1}$ belongs to $c_2$.
If $l=1$ this means that $c_1$ and $c_2$ have a common vertex ($u_0=w_{n-1}$).
If $l\geq 2$, this mean that there is an undirected walk $u_0g_0\ldots g_{l-1}u_{l-1}$ of length $l-1$ joining a vertex of $c_1$ and a vertex of $c_2$.

Either case, if there is an undirected path of length $l$ joining $c_1$ and $c_2$, then there is also one of length $l-1$. By induction on $l$ this shows that $c_1$ and $c_2$ have a common vertex.

Next, up to a relabelling we can assume that the common vertex is $v_0=w_0$. From $r(e_k)=v_0=w_0=r(f_n)$ and the assumption that every vertex receives at most one edge, we deduce that $e_k=f_n$. Iterating the argument we deduce that $c_1=c_2$.
\end{proof}

\begin{ex}[Arborescences]
An \emph{arborescence} is a directed acyclic graph with a distinguished vertex $t$ (called the \emph{root}) such that, for any other vertex $v$, there is exactly one directed walk from $t$ to $v$. In the next picture, the graph on the left is an arborescence but the one on the right is not (but it is a
directed rooted tree):
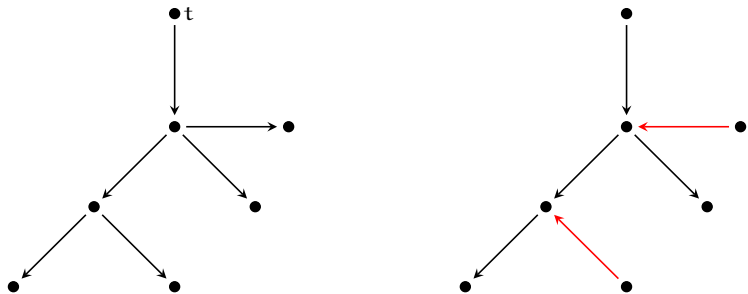
\begin{center}
\begin{tikzpicture}[inner sep=3pt,node distance=1.5cm,font=\footnotesize]

\node[main node] (1) {};
\node (2) [main node,below of=1] {};
\node (3) [main node,right of=2] {};
\node (4) [main node,below right of=2] {};
\node (5) [main node,below left of=2] {};
\node (6) [main node,below right of=5] {};
\node (7) [main node,below left of=5] {};

\draw (1) node[right] {$t$};
\filldraw (1) circle (0.06);
\filldraw (2) circle (0.06);
\filldraw (3) circle (0.06);
\filldraw (4) circle (0.06);
\filldraw (5) circle (0.06);
\filldraw (6) circle (0.06);
\filldraw (7) circle (0.06);

\path[freccia] (1) edge (2) (2) edge (3) (2) edge (4) (2) edge (5) (5) edge (6) (5) edge (7);

\end{tikzpicture}
\hspace{2cm}
\begin{tikzpicture}[inner sep=3pt,node distance=1.5cm]

\node[main node] (1) {};
\node (2) [main node,below of=1] {};
\node (3) [main node,right of=2] {};
\node (4) [main node,below right of=2] {};
\node (5) [main node,below left of=2] {};
\node (6) [main node,below right of=5] {};
\node (7) [main node,below left of=5] {};

\filldraw (1) circle (0.06);
\filldraw (2) circle (0.06);
\filldraw (3) circle (0.06);
\filldraw (4) circle (0.06);
\filldraw (5) circle (0.06);
\filldraw (6) circle (0.06);
\filldraw (7) circle (0.06);

\path[freccia] (1) edge (2) (3) edge[red] (2) (2) edge (4) (2) edge (5) (6) edge[red] (5) (5) edge (7);

\end{tikzpicture}
\end{center}
The undirected graph underlying an arborescence is a rooted tree. Vice versa, a rooted tree becomes an arborescence if we orient all the edges so that the source is closer to the root than the range.

Every arborescence is the transposed of a functional graph (in a tree every vertex $v$ different from the root is preceeded by a unique vertex $w$, the \emph{parent} of $v$).
\end{ex}

\begin{ex}
If $E$ is a directed graph obtained from an arborescence by adding an edge with target the root, then $E$ is the transposed of a functional graph (every vertex including the root receives exactly one edge). An example is in the following picture
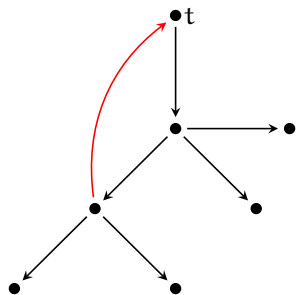
\begin{center}
\begin{tikzpicture}[inner sep=3pt,node distance=1.5cm]

\node[main node] (1) {};
\node (2) [main node,below of=1] {};
\node (3) [main node,right of=2] {};
\node (4) [main node,below right of=2] {};
\node (5) [main node,below left of=2] {};
\node (6) [main node,below right of=5] {};
\node (7) [main node,below left of=5] {};

\draw (1) node[right] {$t$};

\filldraw (1) circle (0.06);
\filldraw (2) circle (0.06);
\filldraw (3) circle (0.06);
\filldraw (4) circle (0.06);
\filldraw (5) circle (0.06);
\filldraw (6) circle (0.06);
\filldraw (7) circle (0.06);

\path[freccia] (1) edge (2) (2) edge (3) (2) edge (4) (2) edge (5) (5) edge (6) (5) edge (7) (5) edge[red,bend left] (1);

\end{tikzpicture}
\end{center}
%Note that in this graph, for every vertex $v\neq t$ there is exactly one directed path from $t$ to $v$. However, it is not an arborescence since it has a directed cycle.
\end{ex}

\begin{ex}[Cycle graphs]
A \emph{cycle graph} is a directed graph whose edges form a directed cycle:
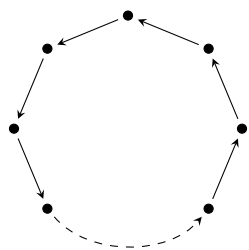
\begin{center}
\begin{tikzpicture}

\clip (-2,-1.8) rectangle (2,1.7);

      \coordinate (1) at (0:1.5);
      \coordinate (2) at (45:1.5);
      \coordinate (3) at (90:1.5);
      \coordinate (4) at (135:1.5);
      \coordinate (5) at (180:1.5);
      \coordinate (6) at (225:1.5);
      \coordinate (7) at (270:1.5);
      \coordinate (8) at (315:1.5);

      \filldraw (1) circle (0.06);
      \filldraw (2) circle (0.06);
      \filldraw (3) circle (0.06);
      \filldraw (4) circle (0.06);
      \filldraw (5) circle (0.06);
      \filldraw (6) circle (0.06);
      \filldraw (8) circle (0.06);

      \path[->,shorten >=4pt,shorten <=4pt] (1) edge (2) (2) edge (3) (3) edge (4) (4) edge (5) (5) edge (6) (8) edge (1) (6) edge[dashed,bend right=50] (8);

\end{tikzpicture}
\end{center}
A cycle graph is its own transposed, and it is a functional graph.
\end{ex}

In view of Prop.~\ref{prop:2} and Prop.~\ref{prop:diagem}, we are interested in connected row-finite graphs $E$ without sinks whose transposed is a functional graph.
Note that there exists arborescences with infinitely many vertices without sinks.
Since we are interested in unital C*-algebras (compact quantum spaces), in the next proposition we focus on finite graphs.

\begin{prop}
Let $E$ be a finite connected graphs without sinks whose transposed is a functional graph.
Then $E$ is a cycle graph.
\end{prop}

\begin{proof}
Let $E$ be a finite connected graphs with no sinks in which every vertex has at most one entry. Every finite acyclic directed graph has at least one sink \cite[pag.~182]{GYZ13}.
Thus $E$ must contain a directed cycle. It follows from Lemma \ref{lemma:arbo} that $E$ contains exactly one directed cycle $c$.
Assume, by contradiction, that $E$ is not a cycle graph.
Then, there is either a vertex $w$ different from those in $c$,
or and edge $f$ different from those in $c$. In the second case, since
the edges of $c$ cannot receive other edges than the ones in $c$, the vertex
$w:=r(f)$ does not belong to $c$. In both cases, since $E$ is connected, there must be an undirected walk joining a vertex in $c$ and $w$. This means that there is at least one edge $f'$ with source in $c$, say $s(f')=v_0$, and range outside the cycle.
Let $F$ be the graph obtained from $E$ by removing all edges in $c$.
Let $F'$ be the connected component of $F$ containing $f'$.
Since $F'$ is acyclic, it has a sink $w'$.
Since $w'$ is not in $c$, there is no edge in $c$ with source $w'$, which means that $w'$ is a sink in $E$ as well, and we get a contradiction.
\end{proof}

The graph C*-algebra of a cycle graph with $n$ edges is $M_n(C(\mathbb{S}^1))$, which is not particularly interesting for us.
Allowing graphs with infinitely many vertices doesn't make things more interesting: indeed, if $E$ is the transposed of a functional graph, $C^*(E)^{U(1)}$ is commutative.

\begin{prop}
Let $E$ be the transposed of a functional graph. Then, $C^*(E)^{U(1)}$ is the diagonal C*-subalgebra of $C^*(E)$.
\end{prop}

\begin{proof}
$C^*(E)^{U(1)}$ is spanned by elements $S_\mu S_\nu^*$ with $r(\mu)=r(\nu)$ and $|\mu|=|\nu|$. Since every vertex receives at most one edge, any two walks with the same range and the same length are equal. Thus, $C^*(E)^{U(1)}$ is spanned by elements $S_\mu S_\mu^*$.
\end{proof}

\section{Line bundles on AF cores}\label{sec:lineAF}

Let $A$ be a unital C*-algebra and $\mathsf{BiVect}(A)$ the collection of $A$-bimodules that are finitely generated and projective as left modules.
The tensor product over $A$ is an internal operation in $\mathsf{BiVect}(A)$ (a proof is spelled out for example in \cite{DM22} or in \cite{DLP24}). On this set we consider two equivalence relations: the isomorphism as bimodules and the isomorphism as left modules, and denote by $\mathsf{BiVect}_{lr}(A)$ and
$\mathsf{BiVect}_l(A)$ the corresponding quotient sets.
%For $M\in\mathsf{BiVect}(A)$, we denote by the same symbol $M$ its class in $\mathsf{BiVect}_{lr}(A)$, and and use the notation $[M]$ for its class in $\mathsf{BiVect}_l(A)$.
Direct sum and tensor product over $A$ induce the structure of a unital semiring on $\mathsf{BiVect}_{lr}(A)$.

An important class of objects in $\mathsf{BiVect}(A)$ is that of self-Morita equivalence $A$-bimodules, \emph{aka} invertible $A$-bimodules. Invertible bimodules are automatically finitely generated and projective as left (and right) modules, and form a group under tensor product $\otimes_A$ called the \emph{Picard group} $\mathsf{Pic}(A)$.
For the Picard group in noncommutative geometry one can see e.g.~\cite{Wal11}, and in particular \cite[Ex.~2.3]{Wal11} for the classical case. Some remarks on the Picard group of a noncommutative algebra are in App.~\ref{sec:app}.

Now, let $E$ be a finite graph with no sinks and $A:=C^*(E)^{U(1)}$.
We can then consider the classes of bimodules $\mathcal{L}_k$ in \eqref{eq:Lk} (``equivariant line bundles'').
Because of \eqref{eq:Picsemiring}, isomorphism classe of these bimodules generate a sub-semiring of $\mathsf{BiVect}_{lr}(A)$, that we shall denote by $\mathsf{SPic}(A)$. For $A:=C^*(E)^{U(1)}$, we have then a sequence of maps
\begin{equation}\label{eq:semigroups}
\mathsf{SPic}(A)\lhook\joinrel\longrightarrow 
\mathsf{BiVect}_{lr}(A)\relbar\joinrel\twoheadrightarrow
\mathsf{BiVect}_l(A)\lhook\joinrel\longrightarrow K_0(A)
\end{equation}
where the first map is an inclusion of semirings, the second one is a surjective homomorphism of semigroups, and the third one is an injective homomorphism of semigroups. Injectivity of the third map follows from the fact that AF algebras have the cancellation property \cite[Ex.~7.4]{RLL00}. Because of the cancellation property, we identify the classes of $M$ in $\mathsf{BiVect}_l(A)$ and in $K_0(A)$ and denote both by $[M]$. If $M\cong A^np$ as left $A$-modules, for some $n\in\Z_+$ and some projection $p\in M_n(A)$, the class in K-theory of such a module will be denoted equivalently by $[M]$ or $[p]$.

\begin{rem}
In general, the map $\mathsf{BiVect}_{lr}(A)\to\mathsf{BiVect}_l(A)$ in \eqref{eq:semigroups} is not injective. A simple counterexample is as follows. We are going to construct a bimodule $N$ that is free
as a left module, but is not free as a bimodule.
Recall that the centralizer of an $A$-bimodule $M$ is the set
\[
Z(M):=\big\{m\in M:am=ma\;\forall\;a\in A\big\} .
\]
If $A$ is commutative, the centralizer is a sub-bimodule of $M$.
In this case, any isomorphism $M\to N$ between $A$-bimodules induces an isomorphism $Z(M)\to Z(N)$ between the centralizers. Now, let $A:=\C^2$ with componentwise multiplication, $M:=\C^2$ with bimodule structure given by the multiplication, and let $N$ be $\C^2$ as a vector space, with left action given by the multiplication, and right action given by
\[
(a_1,a_2)\triangleleft  (b_1,b_2)=(a_1b_2,a_2b_1)
\]
for all $a=(a_1,a_2)\in N$ and $b=(b_1,b_2)\in A$ (so, the right action is the multiplication twisted by an outer automorphism of the algebra). Clearly $Z(M)=A$, since the algebra is commutative.
On the other hand, for $a=(a_1,a_2)\in N$,
\[
(a_1,a_2)\triangleleft (1,0)=(0,a_2) \qquad\text{and}\qquad (1,0)\cdot (a_1,a_2)=(a_1,0)
\]
are equal if and only if $a=0$. Thus $Z(N)=0$ is not isomorphic to $Z(M)$, and $N$ is isomorphic to $M$ as a left module but not as a bimodule.

See Remark \ref{rem:SMEB} for an example of a bimodule of the form \eqref{eq:Lk} that is isomorphic to $\mathcal{L}_0$ as a left $\mathcal{L}_0$-module but not as a bimodule.
\end{rem}

The aim of this section is: to give some sufficient conditions for $K_0(A)$ to be generated by classes of line bundles \eqref{eq:Lk}, to find formulas expressing these classes in terms of classes of the vertex projections of $E$ and of the classes of matrix units in the Bratteli decomposition of $A$, to find a basis of $K_0(A)$ consisting of line bundles in the case it is free and of finite rank.

\subsection{Bratteli diagrams}

Some clue about ``multiplicative structures'' in the K-theory of AF cores comes from a study of their Bratteli diagram.

Let $E$ be a finite graph with no sinks. For $k\geq 1$ and $v\in E^0$, denote by $A_k$ the vector subspace of $C^*(E)^{U(1)}$ spanned of elements $S_\mu S_\rho^*$ with $|\mu|=|\rho|=k-1$ and $r(\mu)=r(\rho)$, and by $A_k(v)$ the subspace spanned by those with $r(\mu)=r(\rho)=v$. Let $A_0$ be the vector subspace spanned by the unit.
Then
\begin{equation}\label{eq:Ak}
A_k=\bigoplus_{v\in E^0}A_k(v)
\end{equation}
is a finite-dimensional unital C*-subalgebra of $C^*(E)^{U(1)}$, and $\bigcup_{k\geq 0}A_k$ is dense in $C^*(E)^{U(1)}$ (see \cite[pag.~27-28]{R05}, but remember the different convention for the Cuntz-Krieger relations).

In general, $A_k(v)$ may be zero for some $k$ and $v$. On the other hand, if $E$ is a graph with no sources, every walk can be extended from the left and $A_k(v)$ is non-zero for every $k$ and $v$.

For each $k\geq 1$ and $v\in E^0$, if $A_k(v)\neq 0$ the set
\[
\big\{ S_\mu S_\rho^* : |\mu|=|\rho|=k-1 , r(\mu)=r(\rho)=v \big\}
\]
is a vector space basis of $A_k(v)$ consisting of matrix units, and $A_k(v)\cong M_{m_{k-1,v}}(\C)$ is a full matrix algebra, where
$m_{k,v}$ is the number of directed walks in $E$ of length $k$ and with range $v$. This number can be computed as follows.

Recall that the \emph{adjacency matrix} of $E$ is that integer matrix $\Gamma=(\Gamma_{v,w})$, with rows and columns labelled by $E^0$, such that $\Gamma_{v,w}$ is the number of edges in $E$ from $v$ to $w$. The number of directed walks from $v$ to $w$ of length $k$ is given by the $(v,w)$ entry of the matrix $\Gamma^k=(\Gamma^{(k)}_{v,w})$. Thus,
\begin{equation}\label{eq:mvk}
m_{k,w}=\sum_{v\in E^0}\Gamma^{(k)}_{v,w} .
\end{equation}
The Bratteli diagram of $C^*(E)^{U(1)}$ can be read from the identity
\begin{equation}\label{eq:itfollows}
S_\mu S_\rho^*=\sum_{e\in E^1:s(e)=r(\mu)}S_\mu S_eS_e^*S_\rho^* ,
\end{equation}
that tells us that $A_k$ is a unital subalgebra of $A_{k+1}$, and $A_k(v)$ embeds in $A_{k+1}(w)$ for every $w\in E^0$ such that there is an edge $e\in E^1$ with source $v$ and range $w$ (the multiplicity of the embedding is the number of edges from $v$ to $w$).

We will draw Bratteli diagrams horizontally, rather than vertically, to save space. We arrange nodes of the diagram in rows and columns, order the vertices, and count the columns starting from $0$. Each row correspond to a vertex in the graph $E$. Each column correspond to a finite-dimensional subalgebra. The node in row $i$ and column $k$ corresponds to the subalgebra $A_k(v_i)$, where $v_i$ is the $i$-th node.

\begin{ex}\label{ex:M2C}
Let $E$ be the following graph:
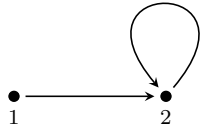
\begin{center}
\begin{tikzpicture}[inner sep=3pt,font=\footnotesize]

\clip (-0.5,-0.5) rectangle (2.5,1.3);

\node[main node] (1) {};
\node (2) [main node,right of=1] {};

\filldraw (1) circle (0.06) node[below=2pt] {$1$};
\filldraw (2) circle (0.06) node[below=2pt] {$2$};

\path[freccia] (2) edge[ciclo] (2);
\path[freccia] (1) edge (2);

\end{tikzpicture}
\end{center}
It is not difficult to show that
$C^*(E)\cong M_2(C(\mathbb{S}^1))$.
The Bratteli diagram of $C^*(E)^{U(1)}$ is:\medskip
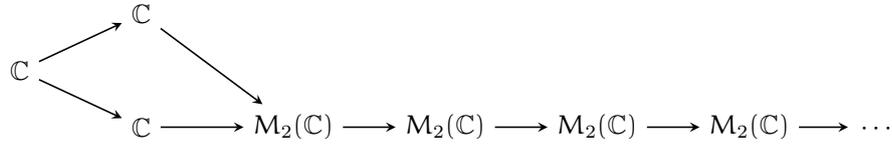
\begin{center}
\begin{tikzpicture}[-stealth,yscale=1.5,xscale=2,semithick]

\node (o) at (0.2,0.5) {$\C$};

\node (a1) at (1,1) {$\C$};
\node (b1) at (1,0) {$\C$};

\foreach \k in {2,...,5} {
	\node (b\k) at (\k,0) {$M_2(\C)$};
}

\coordinate (b6) at (5.65,0);

\node[transform canvas={yshift=-0.7pt}] at ($(b6)+(0.2,0)$) {$\cdots$};

\draw (o) -- (a1);
\draw (o) -- (b1);
\draw (a1) -- (b2);

\foreach[evaluate=\k as \knext using int(\k+1)] \k in {1,...,5} {
	\draw (b\k) -- (b\knext);
}

\end{tikzpicture}
\end{center}
In this example, we write explicitly the matrix algebra corresponding to each node to make it more clear.
The first row corresponds to the vertex $1$, the second to the vertex $2$. We see that the first row ends at the column $1$ (there are not walks of length greater than $0$ with range $1$).
The core is $C^*(E)^{U(1)}\cong M_2(\C)$.
\end{ex}

\begin{ex}\label{ex:againM2C}
Let $E$ be the following graph:\medskip
\begin{center}
\begin{tikzpicture}[inner sep=3pt,font=\footnotesize]

\node[main node] (1) {};
\node (2) [main node,right of=1] {};

\filldraw (1) circle (0.06);
\filldraw (2) circle (0.06);

\path[freccia]
	(1) edge[bend right=20] (2)
	(2) edge[bend right=20] (1);
	
\end{tikzpicture}\medskip
\end{center}
Like in Example \ref{ex:M2C}, one has again $C^*(E)\cong M_2(C(\mathbb{S}^1))$.
But now the Bratteli diagram of $C^*(E)^{U(1)}$ is\medskip\smallskip
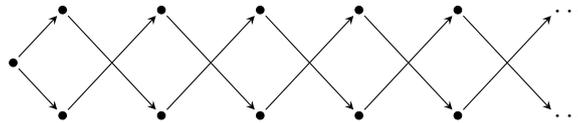
\begin{center}
\begin{tikzpicture}[-stealth,shorten >=3pt,shorten <=3pt]

\coordinate (o) at (0.65,0.7);
\filldraw (o) circle (0.05);

\foreach \k in {1,...,5} {
	\coordinate (a\k) at (1.3*\k,1.4);
	\coordinate (b\k) at (1.3*\k,0);
	\filldraw (a\k) circle (0.05);
	\filldraw (b\k) circle (0.05);
}

\coordinate (a6) at (1.3*6,1.4);
\coordinate (b6) at (1.3*6,0);

\node[transform canvas={yshift=-0.5pt}] at ($(a6)+(0.2,0)$) {$\cdots$};
\node[transform canvas={yshift=-0.5pt}] at ($(b6)+(0.2,0)$) {$\cdots$};

\draw (o) -- (a1);
\draw (o) -- (b1);

\foreach[evaluate=\k as \knext using int(\k+1)] \k in {1,...,5} {
	\draw (a\k) -- (b\knext);
	\draw (b\k) -- (a\knext);
}

\end{tikzpicture}\smallskip
\end{center}
From the Bratteli diagram we recognize that $C^*(E)^{U(1)}\cong\C^2$. (The graph C*-algebra is isomorphic to the one in Example \ref{ex:M2C}, but the isomorphism is not $U(1)$-equivariant as the cores are not isomorphic.)
\end{ex}

\begin{ex}[The $UHF(n^\infty)$ algebra]
Let $n\geq 2$ and let $B_n$ be the graph of the Cuntz algebra in Example~\ref{ex:Cuntz}. The Bratteli diagram of $C^*(B_n)^{U(1)}$ has one node in each column, one arrow between the first two nodes, and $n$ parallel arrows between any other two consecutive nodes
\medskip
\begin{center}
\begin{tikzpicture}[-stealth,shorten >=4pt,shorten <=5pt]

\coordinate (o) at (-1.3*1,0);
\fill (o) circle (0.1);

\foreach \k in {0,...,5} {
	\coordinate (v\k) at (1.3*\k,0);
	\fill (v\k) ellipse (0.08 and 0.2);
}

\coordinate (v6) at (6*1.3,0);

\draw (o) -- (v0);

\node[transform canvas={yshift=-0.5pt}] at ($(v6)+(0.2,0)$) {$\cdots$};

\foreach[evaluate=\k as \knext using int(\k+1)] \k in {0,...,5} {
	\draw[transform canvas={yshift=3pt}] (v\k) -- (v\knext);
	\draw[transform canvas={yshift=0pt}] (v\k) -- (v\knext);
	\draw[transform canvas={yshift=-3pt}]   (v\k) -- (v\knext);
}

\end{tikzpicture}\smallskip
\end{center}
The first node here corresponds to the subalgebra of scalar multiple of the identity. The second node to the subalgebra spanned by $S_\mu S_\rho^*$ with $|\mu|=|\rho|=0$, that is the scalar multiples of the unique vertex projection, which is again the identity element.
\end{ex}

\begin{ex}[The noncommutative space of Penrose tilings]\label{ex:penrose}
Let $\mathcal{P}$ be the graph
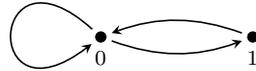
\begin{center}
\begin{tikzpicture}[inner sep=3pt,font=\footnotesize,baseline={([yshift=-3pt]current bounding box.center)}]

\clip (-1.5,-0.6) rectangle (2.4,0.7);

\node[main node] (1) {};
\node (2) [main node,right of=1] {};

\filldraw (1) circle (0.06) node[below=2pt] {$0$};
\filldraw (2) circle (0.06) node[below=2pt] {$1$};

\path[freccia,rotate=90] (1) edge[ciclo] (1);

\path[freccia]
	(1) edge[bend right=20] (2)
	(2) edge[bend right=20] (1);
	
\end{tikzpicture}\smallskip
\end{center}
For reasons that will be clear soon, the vertices of $\mathcal{P}$ are labelled $0$ and $1$.
A simple computation of K-theory using the adjacency matrix gives
$K_0\big(C^*(\mathcal{P})\big)\cong K_1\big(C^*(\mathcal{P})\big)\cong 0$.
Since the adjacency matrix is irreducible and it is not a permutation matrix,
one immediately deduces from \cite[Thm.~6.5]{Ror95} that
\[
C^*(\mathcal{P})\cong\mathcal{O}_2 .
\]
One can write down an explicit isomorphism in terms of covariant and contravariant induction as developed in \cite{HT24cov,HT24} (see also \cite{BS24}). The details can be found in \cite[Example 7.2]{CDH25}.

The Bratteli diagram of $C^*(\mathcal{P})^{U(1)}$ is\medskip\smallskip
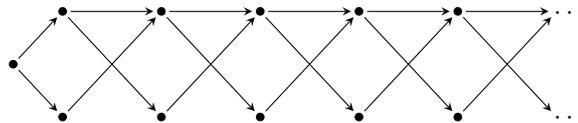
\begin{center}
\begin{tikzpicture}[-stealth,shorten >=3pt,shorten <=3pt]

\coordinate (o) at (0.65,0.7);
\filldraw (o) circle (0.05);

\foreach \k in {1,...,5} {
	\coordinate (a\k) at (1.3*\k,1.4);
	\coordinate (b\k) at (1.3*\k,0);
	\filldraw (a\k) circle (0.05);
	\filldraw (b\k) circle (0.05);
}

\coordinate (a6) at (1.3*6,1.4);
\coordinate (b6) at (1.3*6,0);

\node[transform canvas={yshift=-0.5pt}] at ($(a6)+(0.2,0)$) {$\cdots$};
\node[transform canvas={yshift=-0.5pt}] at ($(b6)+(0.2,0)$) {$\cdots$};

\draw (o) -- (a1);
\draw (o) -- (b1);

\foreach[evaluate=\k as \knext using int(\k+1)] \k in {1,...,5} {
	\draw (a\k) -- (a\knext);
	\draw (a\k) -- (b\knext);
	\draw (b\k) -- (a\knext);
}

\end{tikzpicture}\smallskip
\end{center}
Since for every pair of vertices $v,w\in \mathcal{P}^0$ there is at most one arrow from $v$ to $w$, a right-infinite walk in $\mathcal{P}$ (or in the Bratteli diagram) is simply a sequence of vertices. Looking at the graph, we see that a right-infinite walk is a binary sequence where $1$ is always followed by a $0$. It is well known that $C^*(\mathcal{P})^{U(1)}$ is the groupoid C*-algebra of the relation of tail equivalence on this set of right-infinite walks (it is an \'etale groupoid with the correct topology).
It follows that $C^*(\mathcal{P})^{U(1)}$ is the AF algebra of the noncommutative space parameterizing Penrose tilings, see e.g.~\cite[Example 6.23]{DanPen}.
\end{ex}

\begin{ex}[The quantum projective line]
Let $E=\Sigma_1$ be the graph of $SU_q(2)$ in \eqref{eq:SUq2}.
The Bratteli diagram of $C^*(\Sigma_1)^{U(1)}\cong C(\C P^1_q)$ is\medskip
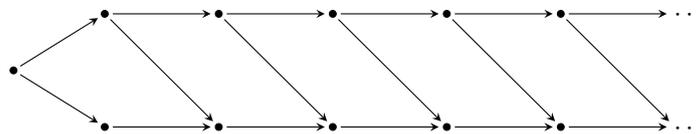
\begin{center}
\begin{tikzpicture}[-stealth,shorten >=3pt,shorten <=3pt,scale=1.5]

\coordinate (o) at (0.2,0.5);
\filldraw (o) circle (0.03);
\node at (6.2,0.5) {};

\foreach \k in {1,...,5} {
	\coordinate (a\k) at (\k,1);
	\coordinate (b\k) at (\k,0);
	\filldraw (a\k) circle (0.03);
	\filldraw (b\k) circle (0.03);
}

\coordinate (a6) at (6,1);
\coordinate (b6) at (6,0);

\node[transform canvas={yshift=-0.5pt}] at ($(a6)+(0.15,0)$) {$\cdots$};
\node[transform canvas={yshift=-0.5pt}] at ($(b6)+(0.15,0)$) {$\cdots$};

\draw (o) -- (a1);
\draw (o) -- (b1);

\foreach[evaluate=\k as \knext using int(\k+1)] \k in {1,...,5} {
	\draw (a\k) -- (a\knext);
	\draw (a\k) -- (b\knext);
	\draw (b\k) -- (b\knext);
}

\end{tikzpicture}\smallskip
\end{center}
The node in column $0$ represents the subalgebra $A_0\cong\C$. The nodes (two) in the $k$-th column, $k\geq 1$, represent the two summands in $A_k$, whose dimension can be computed by looking at the diagram. One finds that $A_k\cong\C\oplus M_k(\C)$, where $\C$ is the top node and $M_k(\C)$ the bottom node. It is well known (and clear from the diagram) that $C(\C P^1_q)$ is the minimal unitization of the C*-algebra of compact operators. (This is the example of non-simple Bratteli diagram that is, e.g., in \cite[Fig.~6.4]{DanPen}.)
\end{ex}

\begin{ex}[Quantum projective spaces]
Let us look at the graph $E=\Sigma_3$ of Example \ref{ex:qsphere}. This should be enough to understand what happens in the general case of $\Sigma_{2n-1}$.
The Bratteli diagram of $C^*(\Sigma_3)^{U(1)}\cong C(\C P^3_q)$ is:\medskip
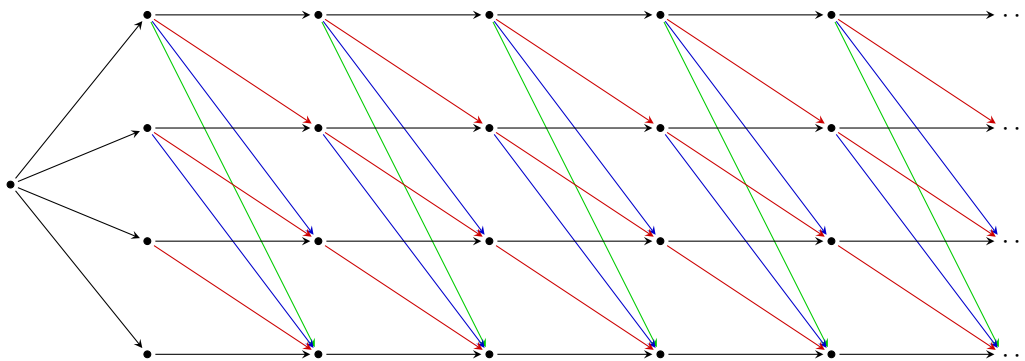
\begin{equation}\label{eq:CP3q}
\begin{tikzpicture}[-stealth,shorten >=3pt,shorten <=3pt,scale=1.5,baseline={([yshift=-6pt]current bounding box.center)}]

\coordinate (o) at (1.5*0.2,1.5);
\node at (1.5*6.2,0.5) {};

\foreach \k in {1,...,5} {
	\coordinate (a\k) at (1.5*\k,3);
	\coordinate (b\k) at (1.5*\k,2);
	\coordinate (c\k) at (1.5*\k,1);
	\coordinate (d\k) at (1.5*\k,0);
}

%\draw[yellow!80!black,very thin,fill=yellow!20] (o) circle (0.2);
%\foreach \k in {1,...,3} {
%\draw[yellow!80!black,very thin,fill=yellow!20] (a\k) circle (0.2);
%}

\filldraw (o) circle (0.03);
\foreach \k in {1,...,5} {
	\filldraw (a\k) circle (0.03);
	\filldraw (b\k) circle (0.03);
	\filldraw (c\k) circle (0.03);
	\filldraw (d\k) circle (0.03);
}

\coordinate (a6) at (1.5*6,3);
\coordinate (b6) at (1.5*6,2);
\coordinate (c6) at (1.5*6,1);
\coordinate (d6) at (1.5*6,0);

\node[transform canvas={yshift=-0.5pt}] at ($(a6)+(0.15,0)$) {$\cdots$};
\node[transform canvas={yshift=-0.5pt}] at ($(b6)+(0.15,0)$) {$\cdots$};
\node[transform canvas={yshift=-0.5pt}] at ($(c6)+(0.15,0)$) {$\cdots$};
\node[transform canvas={yshift=-0.5pt}] at ($(d6)+(0.15,0)$) {$\cdots$};

\draw (o) -- (a1);
\draw (o) -- (b1);
\draw (o) -- (c1);
\draw (o) -- (d1);

\foreach[evaluate=\k as \knext using int(\k+1)] \k in {1,...,5} {
	\draw (a\k) -- (a\knext);
	\draw[red!80!black] (a\k) -- (b\knext);
	\draw[blue!80!black] (a\k) -- (c\knext);
	\draw[green!80!black] (a\k) -- (d\knext);
	\draw (b\k) -- (b\knext);
	\draw[red!80!black] (b\k) -- (c\knext);
	\draw[blue!80!black] (b\k) -- (d\knext);
	\draw (c\k) -- (c\knext);
	\draw[red!80!black] (c\k) -- (d\knext);
	\draw (d\k) -- (d\knext);
}

\end{tikzpicture}\medskip
\end{equation}
In general, the diagram of $C(\C P^{n-1}_q)=C^*(\Sigma_{2n-1})^{U(1)}$ has, besides the root, $n$ rows of nodes and an arrow from the node in row $i$ column $k$ to the node in row $j$ column $k+1$, for all $1\leq i\leq j\leq n$ and all $k\geq 1$. The root corresponds to $A_0\cong\C$, as usual, while the $k$-th node in the $i$-th row corresponds to the subalgebra $A_k(i)$.
The nodes in the bottom row form an hereditary and saturated subset, and the quotient by the corresponding ideal is the map $C(\C P^{n-1})\to C(\C P^{n-2})$.
\end{ex}

\subsection{Generators of K-theory}
In this section, $E$ is a finite graph with no sinks. The subalgebras $A_k$ and $A_k(v)$ of $C^*(E)^{U(1)}$ are those defined in the previous section (note that $A_k(v)$ may be $0$, if there is no walk of length $k-1$ with range $v$).
For each vertex $v$ and each $k\in\N$, choose a walk $\mu_{v,k}$ of length $k$ with range $v$, if it exists, and define the projection
\begin{equation}\label{eq:genQ}
Q_{v,k}:=S_{\mu_{v,k}}S_{\mu_{v,k}}^*  \in A_{k+1}(v) .
\end{equation}
If such a walk doesn't exists, then $A_{k+1}(v)=0$ and we set $Q_{v,k}:=0$.
Note that $Q_{v,0}=P_v$ is the vertex projection of $v\in E^0$.

\begin{prop}\label{prop:generators}
For every $k_0\in\N$, the set
\[
\big\{ [Q_{v,k}]:v\in E^0,k\geq k_0 \big\}
\]
generates $K_0\big(C^*(E)^{U(1)}\big)$.
\end{prop}

\begin{proof}
Recall also that, if $A$ is an AF algebra, then $K_0(A)$ is generated by classes of projections in $A$ (we call them \emph{scalar} projections), cf.~e.g.~\cite{Led14}.
Since $\bigcup_{k\in\N}A_k$ is dense in $A:=C^*(E)^{U(1)}$, for each projection $p_0\in A$ there exists a $k\in\N$ and a projection $p_1\in A_k$ such that $\|p_0-p_1\|<1$.
It follows that $[p_0]=[p_1]$ in $K_0(A)$ (see e.g.~\cite[Prop.~2.2.4]{RLL00}).
Thus, every class $[p_0]\in K_0(A)$ is represented by a projection $p_1\in A_k$ for some $k\in\N$.
Due to the inclusion $A_k\subseteq A_{k+1}$, we can assume that $p_1\in A_k$ for some $k\geq k_0+1$.
It follows that $K_0(A)$ is generated by scalar projections in $\bigcup_{k\geq k_0+1}A_k$.
But $A_k(v)$ is a full matrix algebra for every $v\in E^0$, and for all $[p]\in A_k(v)$ the projection $p$ is equivalent to an integer multiple of a diagonal matrix unit (any diagonal matrix unit in $A_k(v)$).
In particular, $[p]=r[Q_{v,k-1}]$ where $r$ is the rank of $p$.
\end{proof}

\begin{lemma}\label{lemma:Gamma}
For every $v\in E^0$ and $k\in\N$ such that $A_k(v)\neq 0$, one has
\begin{equation}\label{eq:Gamma}
[Q_{v,k}]=\sum_{w\in E^0}\Gamma_{v,w}[Q_{w,k+1}] ,
\end{equation}
where $\Gamma=(\Gamma_{v,w})$ is the adjacency matrix of the graph $E$.
\end{lemma}

\begin{proof}
From \eqref{eq:genQ} and \eqref{eq:itfollows} we get
\[
Q_{v,k}=\sum_{e\in E^1:s(e)=v}S_{\mu_{v,k}}S_eS_e^*S_{\mu_{v,k}}^* .
\]
Since the projections $S_{\mu_{v,k}}S_eS_e^*S_{\mu_{v,k}}^*$ and $S_{\mu_{v,k}}S_fS_f^*S_{\mu_{v,k}}^*$ are orthogonal, for every $e\neq f$, we get
\[
[Q_{v,k}]=\sum_{e\in E^1:s(e)=v}[S_{\mu_{v,k}}S_eS_e^*S_{\mu_{v,k}}^*] .
\]
Next, called $V:=S_{\mu_{v,k}}S_eS_{\mu_{r(e),k+1}}^*$, the projection $VV^*$ is von Neumann
equivalent to
\[
V^*V=S_{\mu_{r(e),k+1}}S_e^*S_{\mu_{v,k}}^*S_{\mu_{v,k}}S_eS_{\mu_{r(e),k+1}}^*=S_{\mu_{r(e),k+1}}S_{\mu_{r(e),k+1}}^*=Q_{r(e),k+1} .
\]
Called $w=r(e)$, we get
\[
[Q_{v,k}]=\sum_{e\in E^1:s(e)=v,r(e)=w}[Q_{w,k+1}] .
\]
The number of edges with source $v$ and range $w$ is by definition $\Gamma_{v,w}$.
\end{proof}

In general, the adjacency matrix need not be invertible (over the integers).
In Example \ref{ex:Cuntz}, $\Gamma=(n)$ is not invertible over the integers if $n>1$.

\begin{prop}\label{prop:Gammainverse}
Let $n:=|E^0|$.
If $\Gamma\in GL(n,\Z)$, then \mbox{$K_0\big(C^*(E)^{U(1)}\big)\cong\Z^n$} is free
and a basis is given by the set $\{P_v:v\in E^0\}$ of classes of vertex projections.
\end{prop}

\begin{proof}
Let $A:=C^*(E)^{U(1)}$. Since $\Gamma$ is invertible, $E$ has no sources (and no sinks), and $A_k(v)$ and $Q_{v,k}$ are non-zero for all $v$ and $k$.

From Prop.~\ref{prop:generators} we know that the set $\big\{ [Q_{w,k}]:w\in E^0,k\in\N \big\}$ generates
$K_0(A)$.
In Lemma~\ref{lemma:Gamma} we can invert the relation and write each $Q_{w,k+1}$ as a combination of $Q_{v,k}$'s and then, by iterating, as a combination of the projections $Q_{v,0}=P_v$.
Thus, $K_0(A)$ is generated by vertex projections. 

The $K_0$-group of an AF algebra is torsion-free (see, e.g., \cite{Led14}), and in the present case since it is finitely generated it must be free.
To show that the vertex projections form a basis, it remains to prove that the rank of $K_0(A)$ is equal to the number $n$ of such projections.

If $\Gamma$ is invertible, for every $k\geq 0$, $\Gamma^k$ is invertible, which means that it has at least one non-zero element in each row and column. Thus, every $v\in E^0$ is the range of at least one walk of length $k$, and $A_k(v)$ in \eqref{eq:Ak} is non-zero. So, $A_k$ has $n$ non-zero summands, and
\[
K_0(A_k)\cong\Z^n , \qquad\forall\;k\in\Z.
\]
To compute the inductive limit, we recall that the map $K_0(A_k)\to K_0(A_{k+1})$ induced by the inclusion $A_k\to A_{k+1}$, under the identification with $\Z^n$, is given by the multiplication by the \emph{matrix of partial multiplicities} of the Bratteli diagram, which in this case is just the adjacency matrix $\Gamma$ of the graph $E$ (see e.g.~\cite[Sect.~6.3.3]{DanPen}). Thus, $K_0(A)$ is the inductive limit of the first row of the following commutative diagram:
\begin{center}
\begin{tikzpicture}[-To]

\node (a0) at (0,1.8) {$\Z^n$};
\node (b0) at (0,0) {$\Z^n$};

\foreach \k in {1,...,4} {
	\node (a\k) at (2*\k,1.8) {$\Z^n$};
	\node (b\k) at (2*\k,0) {$\Z^n$};
}

\node[transform canvas={yshift=-1pt}] (a5) at (9.5,1.8) {$\cdots$};
\node[transform canvas={yshift=-1pt}] (b5) at (9.5,0) {$\cdots$};

\begin{scope}[font=\footnotesize]

\draw[To-] (b0) --node[fill=white]{$\id$} (a0);

\foreach[evaluate=\k as \knext using int(\k+1)] \k in {0,...,3} {
\draw (a\k) --node[above]{$\Gamma$} (a\knext);
\draw (b\k) --node[above]{$\id$} (b\knext);
}

\foreach[evaluate=\k as \knext using int(\k+1)] \k in {1,...,4} {
\draw[To-] (b\k) --node[fill=white]{$\Gamma^{-\k}$} (a\k);
}

\draw (a4) -- (a5);
\draw (b4) -- (b5);
\end{scope}

\end{tikzpicture}
\end{center}
Since the vertical arrows are isomorphisms, the first row and the second row have the same inductive limit, which looking at the second row we realize is $\Z^n$.
\end{proof}

Notice that it is well-known that $K_0(C^*(E))$ is generated by the vertex projections of $E$, cf.~\cite[Prop.~3.8 (1)]{CET12}. But here we are interested in the K-theory of the AF core.
The graph $E=B_n$ of the Cuntz algebra $\mathcal{O}_n$, $n>1$, is an example where $K_0\big(C^*(E)^{U(1)}\big)$ is \emph{not} generated by the vertex projection(s) of the graph $E$.

\subsection{Projections for line bundles}
In this section, $E$ is a finite graph with no sinks, $\mathcal{L}_0=C^*(E)^{U(1)}$, $\mathcal{L}_k$ is the $\mathcal{L}_0$-bimodule in \eqref{eq:Lk}.

\begin{lemma}
Let $W_k$ be the set of directed walks in $E$ of length $k\geq 0$.
Then
\begin{equation}\label{eq:unitB}
\sum_{\mu\in W_k}S_\mu S_\mu^*=1 .
\end{equation}
\end{lemma}

\begin{proof}
Notice that since $E$ has no sinks, $W_k$ is non-empty for every $k$.
The lemma is proved by induction on $k\geq 0$. For $k=0$, \eqref{eq:unitB} becomes
\[
\sum_{v\in E^0}P_vP_v=\sum_{v\in E^0}P_v=1 ,
\]
cf.~\eqref{eq:unitA}. Next, assume that \eqref{eq:unitB} holds for some $k\geq 0$. From \eqref{eq:itfollows},
\begin{align*}
\sum_{\mu\in W_{k+1}}S_\mu S_\mu^* &=
\sum_{\nu\in W_k,e\in E^1,r(\nu)=s(e)}S_\nu S_eS_e^* S_{\nu}^* . \\
\intertext{Since $S_fS_e=0$ if $r(f)\neq s(e)$ (see \cite[Prop.~1.12]{R05}, remembering the different convention on source and range map), one has $S_\nu S_e=0$ for all $\nu$ with $r(\nu)\neq s(e)$, and the last summation is equal to}
&=
\sum_{\nu\in W_k}\sum_{e\in E^1}S_\nu S_eS_e^* S_{\nu}^* . \\
\intertext{Since $E$ has no sinks, for every $v\in E^0$ there is at least one edge with source $v$. Thus, the last summation is equal to}
 &=
\sum_{\nu\in W_k}S_\nu \left(\sum_{v\in E^0}\sum_{e\in s^{-1}(v)}S_eS_e^*\right) S_{\nu}^* \\
\intertext{now we use \eqref{eq:CK3}:}
&=\sum_{\nu\in W_k}S_\nu \left(\sum_{v\in E^0}P_v\right) S_{\nu}^*=\sum_{\nu\in W_k}S_\nu S_{\nu}^* .
\end{align*}
By the inductive hypothesis, last sum is $1$, and this proves the inductive step.
\end{proof}

\begin{prop}\label{prop:linebundlesA}
For all $k\geq 0$, one has
\begin{equation}\label{eq:mkleft}
[\mathcal{L}_{-k}]=\sum_{v\in E^0}m_{k,v}[P_v] ,
\end{equation}
where $m_{k,v}$ is given by \eqref{eq:mvk} (and can be zero for some $k$ and $v$).
\end{prop}

\begin{proof}
Let $W_k$ be the set of directed walks of lengths $k\geq 0$.
Consider the left $\mathcal{L}_0$-module maps
\begin{align*}
\phi_k &:\mathcal{L}_{-k}    \longrightarrow (\mathcal{L}_0)^{W_k} , & a &\longmapsto (aS_{\mu}) , \\
\psi_k &:(\mathcal{L}_0)^{W_k}  \longrightarrow \mathcal{L}_{-k} , & (b_{\mu}) &\longmapsto
\sum_{\mu\in W_k} b_\mu S_\mu^*.
\end{align*}
It follows from \eqref{eq:unitB} that $\psi_k\phi_k$ is the identity on $\mathcal{L}_{-k}$. Thus, one has the isomorphism of left $\mathcal{L}_0$-modules:
\[
\mathcal{L}_{k}\cong\phi_{k}\psi_{k}(\mathcal{L}_0)^{W_k} .
\]
If $|\mu|=|\rho|$, from \eqref{eq:relations} we deduce that $S_\mu^* S_\rho=\delta_{\mu,\rho}P_{r(\mu)}$.
Hence, for $b=(b_\mu)\in (\mathcal{L}_0)^{W_k}$ one has
\[
\phi_{k}\psi_{k}(b)_\mu=\sum_{\rho\in W_k} b_\rho S_\rho^* S_\mu  =b_{\mu}P_{r(\mu)} .
\]
Thus, we have the isomorphisms of left modules
\[
\mathcal{L}_{-k}\cong \bigoplus_{\mu\in W_k}\mathcal{L}_0P_{r(\mu)}\cong
\bigoplus_{v\in E^0}\;\bigoplus_{\substack{\mu\in W_k\\[1pt] r(\mu)=v}}\mathcal{L}_0P_{v}\cong
 \bigoplus_{v\in E^0}(\mathcal{L}_0P_v)^{\oplus m_{k,v}} ,
\]
where we counted the number of walks in $W_k$ with range $v$.
\end{proof}

\begin{prop}\label{prop:linebundlesB}
Assume also that $E$ has no sources. Then, for all $k\geq 0$,
\begin{equation}\label{eq:mkright}
[\mathcal{L}_k]=\sum_{w\in E^0}[Q_{w,k}] .
\end{equation}
\end{prop}

\begin{proof}
For $k=0$ the statement is trivial, so we can assume $k>0$.
Since the graph has no sources, for each $w\in E^0$ we can choose an edge $\eta(w)$ with range $w$.
Let
\begin{equation}\label{eq:elementZ}
Z:=\sum_{w\in E^0}S_{\eta(w)} .
\end{equation}
Clearly
\[
Z^*Z=\sum_{u,w\in E^0}S_{\eta(u)}^*S_{\eta(w)}=\sum_{u,w\in E^0}\delta_{u,w}P_w=1 .
\]
So, for every $k>0$, $(Z^*)^kZ^k=1$ and $Z^k(Z^*)^k$ is a projection.

Since $Z^k\in\mathcal{L}_k$ and $(Z^*)^k\in\mathcal{L}_{-k}$, there are left $\mathcal{L}_0$-module maps
\begin{align*}
\phi_k &:\mathcal{L}_k \longrightarrow \mathcal{L}_0 , & a &\longmapsto a(Z^*)^k , \\
\psi_k &:\mathcal{L}_0 \longrightarrow \mathcal{L}_k , & b &\longmapsto bZ^k .
\end{align*}
Since $\psi_k\phi_k$ is the identity, it follows
$\mathcal{L}_k\cong\phi_k\psi_k(\mathcal{L}_0)=\mathcal{L}_0Z^k(Z^*)^k$
as left $\mathcal{L}_0$-modules. Thus,
\[
[\mathcal{L}_k]=[Z^k(Z^*)^k] .
\]
Consider the subgraph $F$ with $F^0:=E^0$ and $F^1:=\big\{\eta(w):w\in E^0\big\}$. Since every vertex in $F$ has at most one incoming edge, for every vertex $w$ and every $k$ there is at most one directed walk in $F$ with length $k$ and range $w$.

Since $S_eS_f^*=0$ if $r(e)\neq r(f)$, one has
\[
ZZ^*=\sum_{w\in E^0}S_{\eta(w)}S_{\eta(w)}^* .
\]
Since $S_eS_f=0$ if $r(e)\neq s(f)$, by induction on $k\geq 1$ one shows that
\[
Z^k(Z^*)^k=\sum_{w\in E^0}S_{\rho_{w,k}}S_{\rho_{w,k}}^*  ,
\]
where $\rho_{w,k}$ is the unique directed walk in $F$ of length $k$ and range $w$. If there is no such a walk, we set $S_{\rho_{w,k}}:=0$.

The summands in the equation above are orthogonal projections, since the paths $\rho_{w,k}$ have distinct ranges. Thus,
\[
[Z^k(Z^*)^k]=\sum_{w\in E^0}[S_{\rho_{w,k}}S_{\rho_{w,k}}^*] .
\]
We can insert $P_w=S_{\mu_{w,k}}^*S_{\mu_{w,k}}$ in the middle and get
\[
[Z^k(Z^*)^k]=\sum_{w\in E^0}[S_{\rho_{w,k}}S_{\mu_{w,k}}^*S_{\mu_{w,k}}S_{\rho_{w,k}}^*]=\sum_{w\in E^0}[V_{w,k}V_{w,k}^*] ,
\]
where we called $V_{w,k}:=S_{\rho_{w,k}}S_{\mu_{w,k}}^*$. Since $V_{w,k}$ belongs to $C^*(E)^{U(1)}$, $V_{w,k}V_{w,k}^*$ is von Neumann equivalent to
\[
V_{w,k}^*V_{w,k}=
S_{\mu_{w,k}}S_{S_{\rho_{w,k}}}^*
S_{S_{\rho_{w,k}}}S_{\mu_{w,k}}^*=S_{\mu_{w,k}}S_{\mu_{w,k}}^*=Q_{w,k} ,
\]
hence the thesis.
\end{proof}

If the adjacency matrix $\Gamma$ is invertible, we know that the K-theory of the AF core is generated by the vertex projections. There is a simple explicit formula for line bundles in terms of vertex projections.
Let us denote by $\Gamma^{(k)}_{v,w}$ the matrix elements of $\Gamma^k$ as before (for $k\geq 0$ if $\Gamma$ is not invertible, and for any $k\in\Z$ if $\Gamma$ is invertible).

\begin{prop}\label{prop:linebundlesC}
Let $E$ be a finite graph with no sinks. For all $k\geq 0$, one has
\begin{equation}\label{eq:mkagain}
[\mathcal{L}_{-k}]=\sum_{v,w\in E^0}\Gamma^{(k)}_{v,w} [P_w] .
\end{equation}
If $E$ has no sources and $\Gamma$ is invertible over the integers, then \eqref{eq:mkagain} holds for every $k\in\Z$.
\end{prop}

\begin{proof}
The statement for $k\geq 0$ follows from immediately from \eqref{eq:mkleft} and \eqref{eq:mvk}.
If $\Gamma$ is invertible, we can invert the relation \eqref{eq:Gamma} and get
\[
[Q_{v,j+1}]=\sum_{w\in E^0}\Gamma_{v,w}^{(-1)}[Q_{w,j}]
\]
for all $j\geq 0$. By induction
\[
[Q_{v,j+k}]=\sum_{w\in E^0}\Gamma_{v,w}^{(-k)}[Q_{w,j}]
\]
for all $j\geq 0$ and $k>0$. If $E$ has no sources, using this with $j=0$ in \eqref{eq:mkright} we get \eqref{eq:mkagain} for the line bundles of positive degree.
\end{proof}

Let $n:=|E^0|$ and number the vertices from $1$ to $n$. So, $E^0=\{1,\ldots,n\}$.
If $\Gamma$ is invertible, we can define $m_{k,j}$ as in \eqref{eq:mvk} for any $j\in E^0$ and $k\in\Z$, and 
rewrite \eqref{eq:mkagain} as
\begin{equation}\label{eqLkpj}
[\mathcal{L}_k]=\sum_{j=1}^n  m_{-k,j} [P_j]  .
\end{equation}
Under the assumptions of Prop.~\ref{prop:Gammainverse}, if the infinite matrix $(m_{-k,j})$ has an $n\times n$ submatrix that is invertible over the integers, then the line bundles corresponding to the rows of such a submatrix generate $K_0\big(C^*(E)^{U(1)}\big)$.

\smallskip

We now give two examples. The first one shows that the invertibility of $\Gamma$ in Prop.~\ref{prop:Gammainverse} is a sufficient but not a necessary condition for the vertex projections to generate the K-theory. The second one is an example where $\Gamma$ is invertible but the matrix $(m_{-k,j})$ is not full rank. In both examples, the line bundles \eqref{eq:Lk} are all trivial and no family of these line bundles can generate the K-theory.

\begin{ex}
Let $E$ be the graph in Example \ref{ex:M2C}.
The adjacency matrix
\[
\Gamma=\mat{0 & 1 \\ 0 & 1}
\]
is not invertible. From \eqref{eq:Gamma}, $[P_1]=[P_2]=[Q_{2,k}]$ and $Q_{1,k}=0$ for all $k\geq 1$.

The singleton $\{[P_1]\}$ is a basis of $K_0(M_2(\C))\cong\Z$ (cf.~Prop.~\ref{prop:generators}).
Since a matrix algebra has only the trivial self-Morita equivalence bimodule (see Appendix \ref{sec:app} for the case of a general finite-dimensional C*-algebra), then $\mathcal{L}_k$ is isomorphic to $\mathcal{L}_0$ as a bimodule and
\[
[\mathcal{L}_k]=[1]=[P_1]+[P_2]=2[P_2] , \qquad\quad\forall\;k\in\Z .
\]
In particular, line bundles generate only the K-theory classes that are divisible by $2$.
\end{ex}

\begin{ex}
Let $E$ be the graph in Example \ref{ex:againM2C}. The adjacency matrix
\[
\Gamma=\mat{0 & 1 \\ 1 & 0}
\]
is invertible, but the matrix $(m_{-k,j})$ has all rows equal to $(1,1)$, so it's not full rank.

From \eqref{eq:mkagain} we see that $[\mathcal{L}_k]=[P_1]+[P_2]=[1]$ for all $k\in\Z$,
so that it is not possible to generate $K_0(\C^2)\cong\Z^2$ only with the line bundles \eqref{eq:Lk}.
Since $\C^2$ has a non-trivial invertible bimodule (cf.~Appendix \ref{sec:app}), this also shows that not all invertible bimodules are of the form \eqref{eq:Lk}.
\end{ex}

\section{Applications}\label{sec:appli}

Let $E^0=\{1,\ldots,n\}$ like in the previous section. Let $\Gamma$ be the adjacency matrix of $E$ and let $c_0,c_1,\ldots,c_n\in\Z$ be given by
\begin{equation}\label{eq:monic}
\sum_{k=0}^nc_kt^k:=\det (t\Gamma-1) .
\end{equation}
Note that $c_n=\det(\Gamma)$ and $c_0=(-1)^n$.

\begin{prop}[Atiyah-Todd identities]~\label{cor:AT}
Let $E$ be a finite graph with invertible adjacency matrix $\Gamma$, and let $c_k$ be the coefficients in \eqref{eq:monic}. Then, the following identities hold in $K_0\big(C^*(E)^{U(1)}\big)$:
\begin{subequations}\label{eq:ATq}
\begin{alignat}{2}
[\mathcal{L}_k] &=-c_n\sum_{j=k-n}^{k-1}c_{j-(k-n)}[\mathcal{L}_j] , \qquad\quad && \forall\;k\geq n, \label{eq:ATq1} \\
[\mathcal{L}_{-k}] &=(-1)^{n+1}\sum_{j=-k+1}^{-k+n}c_{j+k}[\mathcal{L}_j] ,
&&\forall\;k\geq 1 . \label{eq:ATq2}
\end{alignat}
\end{subequations}
\end{prop}

\begin{proof}
Note that if $\Gamma$ is invertible, $E$ has no sinks and no sources
and $c_n=\det(\Gamma)=\pm 1$.
From
\[
\det (t\Gamma-1)=\det(\Gamma)\det (t 1-\Gamma^{-1})
\]
we see that, up to a sign, \eqref{eq:monic} is the characteristic polynomial of $\Gamma^{-1}$. It follows from Cayley-Hamilton theorem that
\[
\sum_{k=0}^nc_k\Gamma^{-k}=0
\]
and multiplying by $\Gamma^{-a}$ we find
\[
\sum_{k=0}^nc_k\Gamma^{-k-a}=0, \qquad\quad\forall\;a\in\Z.
\]
In terms of matrix elements,
\[
\sum_{k=0}^nc_k \Gamma^{(-k-a)}_{i,j}=0 \;,\qquad\quad\forall\; i,j\in\{1,\ldots,n\}.
\]
Using \eqref{eq:mkagain} we find
\[
\sum_{k=0}^nc_k[\mathcal{L}_{k+a}]=\sum_{i,j=1}^n
\Big(\sum_{k=0}^nc_k\Gamma^{(-k-a)}_{i,j}\Big) [P_j] =0.
\]
If in the above sum we single out the term with $k=n$, after a reparametrization we find \eqref{eq:ATq1}. If we single out the term with $k=0$, after a reparametrization we find \eqref{eq:ATq2}.
\end{proof}

\begin{prop}\label{prop:ring}
Let $E$ be a finite graph, 
let $\Gamma$ be its adjacency matrix,
$n:=|E^0|$, and $m_{k,j}$ the integers in \eqref{eq:mvk}.
Assume that both $\Gamma$ and the matrix
\begin{equation}\label{eq:matrixM}
\mathcal{M}:=\begin{pmatrix}
m_{0,1} & m_{0,2} & \cdots & m_{0,n} \\
m_{-1,1} & m_{-1,2} & \cdots & m_{-1,n} \\
\vdots & \vdots && \vdots \\
m_{-n+1,1} & m_{-n+1,2} & \cdots & m_{-n+1,n}
\end{pmatrix}
\end{equation}
are invertible. Then:\smallskip
\begin{enumerate}\itemsep=5pt
\item\label{en:ringA}
The set $\bigl\{[\mathcal{L}_k]:0\leq k<n\bigr\}$ is a basis of $K_0\big(C^*(E)^{U(1)}\big)$.

\item\label{en:ringB}
There exists an isomorphism
\[
\varphi:K_0\big(C^*(E)^{U(1)}\big)\to\Z[\lambda]/\bigl(\det (\lambda\Gamma-1) \bigr)
\]
of abelian groups given on classes of line bundles by
\begin{equation}\label{eq:Lklambdak}
\varphi([\mathcal{L}_k])= \lambda^k  \quad\forall\;k\in\Z .
\end{equation}
(We shall use the same symbol for an element in $\Z[\lambda]$ and its class in the quotient ring.)

\item\label{en:ringC}
The map
\[
\psi:\mathsf{SPic}\big(C^*(E)^{U(1)}\big)\to\Z[\lambda]/\bigl(\det (\lambda\Gamma-1) \bigr)
\]
obtained as a composition of $\varphi$ with \eqref{eq:semigroups} is a homomorphism of semirings.
\end{enumerate}
\end{prop}

\begin{proof}
Point \ref{en:ringA} follows from the discussion after Prop.~\ref{prop:linebundlesC}.

\smallskip

Call $J$ the ideal generated by $\det (\lambda\Gamma-1)$.
Since $\Gamma$ is invertible over the integers, it is a monic polynomial (up to a sign). Thus $\Z[\lambda]/J$ is a free $\Z$-module of rank $n$ and is spanned by powers $\lambda^k$, $k=0,\ldots,n-1$.
Since the set at point \ref{en:ringA} is a basis of the K-group, there is an isomorphism
$\varphi:K_0\big(C^*(E)^{U(1)}\big)\to\Z[\lambda]/\bigl(\det (\lambda\Gamma-1) \bigr)$ defined in terms of these bases by
\begin{equation}\label{eq:itsdef}
\varphi([\mathcal{L}_k]):= \lambda^k  \quad\forall\;k\in\{0,\ldots,n-1\}.
\end{equation}
We now show that the isomorphism thus defined satisfies \eqref{eq:Lklambdak} using the Atiyah-Todd identities.

First we notice that the powers $\lambda^k$ satisfy the same identities as $[\mathcal{L}_k]$, coming from the fact that $\det(\lambda\Gamma-1)=0$ in the quotient ring:
\begin{subequations}
\begin{alignat}{2}
\lambda^k &=-c_n\sum_{j=k-n}^{k-1}c_{j-(k-n)}\lambda^j , \qquad\quad && \forall\;k\geq n, \label{eq:ATql1} \\
\lambda^{-k} &=(-1)^{n+1}\sum_{j=-k+1}^{-k+n}c_{j+k}\lambda^j ,
&&\forall\;k\geq 1 . \label{eq:ATql2}
\end{alignat}
\end{subequations}
Now we can prove \eqref{eq:Lklambdak} by induction on $k$ both for $k\geq n-1$ and $k\leq 0$.

For $j\in\{0,\ldots,n-1\}$ the identity holds by definition of $\varphi$. Let $k\geq n$ and suppose that \eqref{eq:Lklambdak} holds for all $j\in\{0,\ldots,k-1\}$. Then,
applying $\varphi$ to both sides of \eqref{eq:ATq1}, using the inductive hypothesis and 
\eqref{eq:ATql1} we find
\[
\varphi([\mathcal{L}_k]) \stackrel{\eqref{eq:ATq1}}{=}
-c_n\sum_{j=k-n}^{k-1}c_{j-(k-n)}\varphi([\mathcal{L}_j] )
=-c_n\sum_{j=k-n}^{k-1}c_{j-(k-n)}\lambda^j
\stackrel{\eqref{eq:ATql1}}{=}
\lambda^k .
\]
So, the equality \eqref{eq:Lklambdak} holds for $j=k$ as well, thus proving the inductive step for positive $k$'s.

Next, let $k<0$ and suppose that \eqref{eq:Lklambdak} holds for all $j\in\{k+1,k+2,\ldots,k+n\}$. Then,
applying $\varphi$ to both sides of \eqref{eq:ATq2}, using the inductive hypothesis and 
\eqref{eq:ATql2} we find
\[
\varphi([\mathcal{L}_k]) \stackrel{\eqref{eq:ATq2}}{=}
(-1)^{n+1}\sum_{j=k+1}^{k+n}c_{j-k}[\mathcal{L}_j]
=(-1)^{n+1}\sum_{j=k+1}^{k+n}c_{j-k}\lambda^j
\stackrel{\eqref{eq:ATql2}}{=}
\lambda^k .
\]
So, the equality \eqref{eq:Lklambdak} holds for $j=k$ as well, thus proving the inductive step for negative $k$'s.
This concludes the proof of \ref{en:ringB}.

\smallskip

By construction, $\mathsf{SPic}(A)$ is a quotient of $\mathbb{N}[\lambda,\lambda^{-1}]$. The quotient map $\pi:\mathbb{N}[\lambda,\lambda^{-1}]\to \mathsf{SPic}(A)$ sends
$\lambda^k$ to the bimodule class of $\mathcal{L}_k$ for all $k\in\Z$.
But it follows from \eqref{eq:Lklambdak} that
\[
\psi\circ\pi(\lambda^k)=\lambda^k \qquad\quad\forall\;k\in\Z
\]
(where the one on the right hand side is the class of $\lambda^k$ in the quotient ring).
The composition $\psi\circ\pi:\mathbb{N}[\lambda,\lambda^{-1}]\to\Z[\lambda]/J$ is clearly a homomorphism of semirings. Since $\pi$ is a quotient map (of semirings), we deduce that $\psi$ is a homomorphism of semirings as well.
\end{proof}

We will see, now, two applications of Prop.~\ref{prop:ring}:
to the C*-algebra of Penrose tilings $C^*(\mathcal{P})^{U(1)}$, where
$\mathcal{P}$ is the graph of Example \ref{ex:penrose}, and to quantum complex projective spaces
$C(\C P^{n-1}_q)=C^*(\Sigma_{2n-1})^{U(1)}$, where $\Sigma_{2n-1}$ is the graph in Example \ref{ex:qsphere}.
Finally, using the results in Sect.~\ref{sec:lineAF} we will see that a proposition close to Prop.~\ref{prop:ring} can be proved also for the $UHF(n^\infty)$ algebra $C^*(B_n)^{U(1)}$, where $B_n$ be the graph of the Cuntz algebra in Example~\ref{ex:Cuntz} (even if the adjacency matrix is not invertible in this case).

\subsection{The C*-algebra of Penrose tilings}\label{sec:penrose}
Let $\mathcal{P}$ be the graph of Example \ref{ex:penrose}.
Let us label the vertices $\{1,2\}$ instead of $\{0,1\}$, to be consistent.

\begin{prop}
$\mathcal{P}$ satisfies the hypothesis of Prop.~\ref{prop:ring}.
\end{prop}

\begin{proof}
We compute the adjacency matrix, its inverse, and the matrix \eqref{eq:matrixM}:
\[
\Gamma=\mat{1 & 1 \\ 1 & 0}
\quad\Longrightarrow\quad
\Gamma^{-1}=\mat{0 & 1 \\ 1 & \!\!-1}
\quad\Longrightarrow\quad
\mat{ m_{0,1} & m_{0,2} \\ m_{-1,1} & m_{-1,2} }=\mat{ 1 & 1 \\ 1 & 0  } .
\]
Invertibility of $\Gamma$ and $\mathcal{M}$ tell us that the hypothesis of Prop.~\ref{prop:ring} are satisfied.
\end{proof}

From the above formulas for $m_{-1,i}$ and \eqref{eqLkpj} we see that
\[
[\mathcal{L}_1]=[P_1] .
\]
The ring \eqref{eq:multmor} is the ring of polynomials in one variable evaluated at the \emph{golden ratio}:
\[
\Z[\lambda]/(\lambda^2+\lambda-1)\cong\Z[\tfrac{1+\sqrt{5}}{2}]
\]
(the isomorphism is given by $\lambda^{-1}\mapsto\frac{1+\sqrt{5}}{2}$).

We can explicitly write the classes of line bundles in the basis $\big\{[\mathcal{L}_0],[\mathcal{L}_1]\big\}$. For all $k\geq 0$:
\[
\Gamma^k=\mat{F_{k+1} & F_k \\ F_k & F_{k-1}} \qquad\Longrightarrow\qquad
\Gamma^{-k}=(-1)^k \mat{F_{k-1} & -F_k \\ -F_k & F_{k+1}} ,
\]
where $F_j$ is the $j$-th \emph{Fibonacci number} (the formula on the left is well known, and the inverse is computed using $\det(\Gamma^k)=(\det \Gamma)^k=(-1)^k$). Using the recursive formula $F_j+F_{j-1}=F_{j+1}$ and \eqref{eqLkpj} we find the explicit expression of the Atiyah-Todd identities in the present example, which is:
\begin{subequations}\label{eq:LFibo}
\begin{align}
[\mathcal{L}_k] &=(-1)^k (F_{k-1}[\mathcal{L}_0]-F_k[\mathcal{L}_1]) ,\\
[\mathcal{L}_{-k}] &=F_{k+1}[\mathcal{L}_0]+F_k[\mathcal{L}_1] ,
\end{align}
\end{subequations}
for all $k\geq 1$.

\subsection{Quantum projective spaces}\label{sec:6}
In this section, $\Sigma_{2n-1}$ is the graph in Example \ref{ex:qsphere}.
The K-theory of $C^*(\Sigma_{2n-1})^{U(1)}=C(\C P^{n-1}_q)$ can be computed in many different ways.
It was originally computed in \cite{HS02} using the fact that $C(\C P^{n-1}_q)$ is itself a graph C*-algebra.
Another possibility is to use its realization as a groupoid C*-algebra \cite{She19}, or one can use
the six-term exact sequence associated to the extension
\[
0\to \mathcal{K}\to C(\C P^{n-1}_q)\to C(\C P^{n-2}_q)\to 0 ,
\]
that can be proved directly by studying the irreducible representations of the algebra, once this is defined by generators and relations. Finally, one can use the Mayer-Vietoris exact sequence associated to the CW structure of $\C P^{n-1}_q$ \cite{ADHT22}. Either way, one finds that $K_0\big(C(\C P^{n-1}_q)\big)$ is the free abelian group $\Z^n$. A basis in terms of matrix units or of line bundles was derived in \cite{She19}.
The identity \eqref{eq:AT1} for quantum projective spaces where originally proved in \cite{ABL14} using an index argument, and then again in \cite{DHMSZ20} using groupoid C*-algebras.

We take the opportunity here to clarify a detail that was overlooked in \cite{ADHT22}.
In \cite[Sect.~4.1.3]{ADHT22}, it is claimed that $K_0\big(C(\C P^{n-1}_q)\big)$ is generated by $\{[P_1],\ldots,[P_n]\}$ and that this simply follows from \cite[Prop.~3.8 (1)]{CET12}, which says that for a graph $E$, an explicit set of generators for $K_0(C^*(E))$ is given by the classes of the vertex projections in $E$. However, here $P_1,\ldots,P_n$ here are the vertex projections of the graph $\Sigma_{2n-1}$ of $\mathbb{S}^{2n-1}_q$, not of the graph of $\C P^{n-1}_q$. To complete the proof, one should argue that the embedding $C(\C P^{n-1}_q)\to C(\mathbb{S}^{2n-1}_q)$, in terms of Cuntz-Krieger presentations of $\C P^{n-1}_q$ and $\mathbb{S}^{2n-1}_q$, sends the vertex projections of the graph of $\C P^{n-1}_q$ to the vertex projections of the graph of $\mathbb{S}^{2n-1}_q$.

All these results are re-derived here as a corollary of the main theory of Sect.~\ref{sec:lineAF}. In particular, the fact that $K_0\big(C(\C P^{n-1}_q)\big)$ is generated by the vertex projections of $\Sigma_{2n-1}$ follows from the invertibility of the adjacency matrix, cf.~Prop.~\ref{prop:Gammainverse}.

\begin{prop}
$\Sigma_{2n-1}$ satisfies the hypothesis of Prop.~\ref{prop:ring}.
\end{prop}

\begin{proof}
The adjacency matrix $\Gamma$ is the uppertriangular matrix with $1$ on the diagonal and above, and zero everywhere else.
Let $m_{k,j}$ be the integers in \eqref{eq:mvk} and consider the matrix
\[
\mathcal{M}':=\begin{pmatrix}
m_{-1,1} & m_{-1,2} & \cdots & m_{-1,n} \\
m_{-2,1} & m_{-2,2} & \cdots & m_{-2,n} \\
\vdots & \vdots && \vdots     \\
m_{-n,1} & m_{-n,2} & \cdots & m_{-n,n}
\end{pmatrix}
\]
It follows from \eqref{eq:mvk} that
\begin{equation}\label{eq:recurrence}
m_{k+1,j}=\sum_{l=1}^n\Gamma^{(k+1)}_{l,j}=\sum_{i,l=1}^n\Gamma^{(k)}_{l,i}\Gamma_{i,j}
=\sum_{i=1}^n m_{k,i}\Gamma_{i,j}=\sum_{i=1}^jm_{k,i} ,
\end{equation}
for all $k\in\Z$. In particular, it follows that
\[
\mathcal{M}=\mathcal{M}'\cdot\Gamma
\]
where $\mathcal{M}$ is the matrix \eqref{eq:matrixM}. 
Since $\Gamma$ is clearly invertible, to complete the proof, it is enough to show that $\mathcal{M}'$ is invertible.

First, we prove by induction on $k$ that
\begin{equation}\label{eq:GammakCPnq}
\Gamma^{(-k)}_{i,j}=(-1)^{j-i}\binom{k}{j-i} ,\qquad\quad\forall\;k\geq 1 .
\end{equation}
The inverse of the adjacency matrix is
\[
\Gamma^{-1}=\begin{pmatrix}
1 & -1 & 0 & 0 & \ldots & 0 \\
0 & 1 & -1 & 0 & \ldots & 0 \\
0 & 0 & 1 & -1 & \ldots & 0 \\
0 & 0 & 0 & 1 & \ldots & 0 \\
\vdots & \vdots & \vdots & \vdots & & \vdots \\
0 & 0 & 0 & 0 & \ldots & 1
\end{pmatrix}
\]
so the formula is true for $k=1$. From the recurrence formula,
\begin{align*}
\Gamma^{(-k-1)}_{i,j} &=\sum_{l=1}^n\Gamma^{(-k)}_{i,l}\Gamma^{-1}_{l,j}=
\Gamma^{(-k)}_{i,j}-\Gamma^{(-k)}_{i,j-1} \\
&=(-1)^{j-i}\left\{\binom{k}{j-i}+\binom{k}{j-i-1}\right\}=
(-1)^{j-i}\binom{k+1}{j-i} ,
\end{align*}
we deduce the inductive step. From \eqref{eq:GammakCPnq} we get
\begin{equation}\label{eq:weseefrom}
m_{-k,j}=
\sum_{i=1}^n(-1)^{j-i}\binom{k}{j-i} =
\sum_{h=0}^{j-1}(-1)^h\binom{k}{h}=(-1)^{j-1}\binom{k-1}{j-1} .
\end{equation}
For $k\in\{1,\ldots,n\}$, this is the element of $\mathcal{M}'$ in position $(k,j)$.
We claim that the matrix with $(i,k)$ element given by
$(-1)^{k-1}\binom{i-1}{k-1}$ is the inverse of $\mathcal{M}'$. This follows from the well-known identity
\[
\sum_{k=1}^n(-1)^{j+k}\binom{i-1}{k-1}\binom{k-1}{j-1}=
\delta_{i,j} . \vspace*{-12pt}
\]
\end{proof}

From \eqref{eq:weseefrom} and \eqref{eqLkpj} we see that
\[
[\mathcal{L}_1]=[P_1] ,
\]
like in the example of Sect.~\ref{sec:penrose}.

The ring \eqref{eq:multmor} in the present example is the ring of truncated polynomials:
\[
\Z[\lambda]/( (\lambda-1)^n )\cong \Z[x]/(x^n) ,
\]
as expected. Note that the inverse of the isomorphism \eqref{eq:Lklambdak}
maps $x=1-\lambda$ to the Euler class $[1]-[\mathcal{L}_1]$ of the line bundle
$\mathcal{L}_1$, as in the classical case.

The coefficients in \eqref{eq:monic} are easily computed using the binomial theorem
\[
\det (t \Gamma-1)=(t-1)^n=\sum_{k=0}^n\binom{n}{k}(-1)^{n-k}t^k .
\]
Inserting them in Prop.~\ref{cor:AT} we find the following corollary.

\begin{cor}[Atiyah-Todd for $\C P^{n-1}_q$]
In $K_0\big(C(\C P^{n-1}_q)\big)$ the following identities hold:
\begin{align*}
[\mathcal{L}_n] &=(-1)^{n+1}\sum_{j=0}^{n-1}\binom{n}{j}(-1)^j[\mathcal{L}_j] , \\
[\mathcal{L}_{-1}] &=\sum_{j=0}^{n-1}\binom{n}{j+1}(-1)^{j}[\mathcal{L}_j] .
\end{align*}
\end{cor}

These are the identities proved in \cite{DHMSZ20} using groupoid C*-algebras, the first one being originally proved in \cite{ABL14}.

Using \eqref{eq:weseefrom} and \eqref{eqLkpj} one can write down the explicit relations between classes of line bundles and vertex projections:
\[
[\mathcal{L}_k]=\sum_{j=1}^n(-1)^{j-1}\binom{k-1}{j-1}[P_j] ,\qquad\quad\forall\;k\geq 1 .
\]
To find a similar relation for line bundles of negative degree we still have to compute $m_{k,j}$ for $k\geq 1$. This is done in the next lemma.

\begin{lemma}
For all $k\geq 0$ and all $1\leq j\leq n$, one has
\begin{equation}\label{eq:mikformula}
m_{k,j}=\binom{j+k-1}{k} .
\end{equation}
\end{lemma}

\begin{proof}
We are going to need the second to last identity in \cite[Table 174]{GKP94}, which we copy here. For all integers $b\geq a\geq 0$, one has
\begin{equation}\label{eq:binomialA}
\sum_{h=a}^b\binom{h}{a}=\binom{b+1}{a+1} .
\end{equation}
Now we can prove \eqref{eq:mikformula} by induction on $k$.
Since $m_{0,j}=1$ for all $j$, for $k=0$ the statement is true.
Assume that \eqref{eq:mikformula} is true for some $k\geq 0$. Then
\[
m_{k+1,j}
\stackrel{\eqref{eq:recurrence}}{=}
\sum_{i=1}^jm_{k,i}=\sum_{i=1}^j\binom{i+k-1}{k}
\stackrel{\eqref{eq:binomialA}}{=}
\binom{j+k}{k+1} .
\]
This proves the inductive step.
\end{proof}

Using \eqref{eq:mikformula} and \eqref{eqLkpj} we can now write the last identity:
\[
[\mathcal{L}_{-k}] =\sum_{j=1}^n\binom{j+k-1}{k}[P_j] , \qquad\quad  \forall\;k\geq 0.
\]

\subsection{The $UHF(n^\infty)$ algebra}

The adjacency matrix of the graph $B_n$ of the Cuntz algebra $\mathcal{O}_n$
is given by $\Gamma=(n)$ and is not invertible over the integers, so that the hypothesis of Prop.~\ref{prop:ring} are not satisfied. Nevertheless, we can still prove the following theorem, analogous to Prop.~\ref{prop:ring}.

\begin{prop}~\label{prop:ringCuntz}
\begin{enumerate}
\item\label{en:ringCuntzA}
For all $k_0\in\Z$, the set $\bigl\{[\mathcal{L}_k]:k\geq k_0\bigr\}$
generates $K_0\big(C^*(B_n)^{U(1)}\big)$.

\item\label{en:ringCuntzB}
There exists an isomorphism
\[
\varphi:K_0\big(C^*(B_n)^{U(1)}\big)\to\Z[\tfrac{1}{n}]
\]
of abelian groups given on classes of line bundles by
\begin{equation}\label{eq:LklambdakCuntz}
\varphi([\mathcal{L}_k])= \lambda^k  \quad\forall\;k\in\Z ,
\end{equation}
where $\lambda:=1/n$.

\item\label{en:ringCuntzC}
The map
\[
\psi:\mathsf{SPic}\big(C^*(B_n)^{U(1)}\big)\to\Z[\tfrac{1}{n}]
\]
obtained as a composition of $\varphi$ with \eqref{eq:semigroups} is a homomorphism of semirings.
\end{enumerate}
\end{prop}

Here we denote by $\Z[\tfrac{1}{n}]$ the ring of integer polynomials in one variable evaluated at $1/n$.
Of course, it is already well-known that the K-theory of the $UHF(n^\infty)$ algebra is
given by $\Z[\tfrac{1}{n}]$.
Here we show that this isomorphism of abelian groups is compatible with the multiplicative structure of line bundles.
Observe that, in complete analogy with Prop.~\ref{prop:ring}:
\[
\Z[\tfrac{1}{n}]\cong\Z[\lambda]/(n\lambda-1)=\Z[\lambda]/\bigl(\det (\lambda\Gamma-1) \bigr) .
\]

\begin{proof}[Proof of Prop.~\ref{prop:ringCuntz}]
Consider the commutative diagram
\begin{equation}\label{eq:indlimit}
\begin{tikzpicture}[-To,baseline=(current bounding box.center)]

\foreach \k in {0,...,4} \node (a\k) at (1.5*\k,1.4) {$\Z$};

\node (b0) at (1.5*0,0) {$\Z$};
\node (b1) at (1.5*1,0) {$\frac{1}{n}\Z$};
\node (b2) at (1.5*2,0) {$\frac{1}{n^2}\Z$};
\node (b3) at (1.5*3,0) {$\frac{1}{n^3}\Z$};
\node (b4) at (1.5*4,0) {$\frac{1}{n^4}\Z$};

\node[transform canvas={yshift=-1pt}] (a5) at (1.5*5,1.4) {$\cdots$};
\node[transform canvas={yshift=-1pt}] (b5) at (1.5*5,0) {$\cdots$};

\begin{scope}[font=\footnotesize]
\draw (a0) --node[left] {$ 1$} (b0);
\draw (a1) --node[left] {$\frac{1}{n}$} (b1);
\draw (a2) --node[left] {$\frac{1}{n^2}$} (b2);
\draw (a3) --node[left] {$\frac{1}{n^3}$} (b3);
\draw (a4) --node[left] {$\frac{1}{n^4}$} (b4);

\foreach[evaluate=\k as \knext using int(\k+1)] \k in {0,...,3} {
\draw (a\k) --node[above]{$ n$} (a\knext);
\draw (b\k) --node[above]{$ 1$} (b\knext);
}
\end{scope}

\draw (a4) -- (a5);
\draw (b4) -- (b5);

\end{tikzpicture}
\end{equation}
The first row is the inductive sequence in K-theory induced by the inductive sequence $A_0\to A_1\to A_1\to\ldots$ of subalgebras of $A$.
By an arrow ``$\xrightarrow{\;\lambda\;}$'' here we mean the map given by multiplication by $\lambda$. Since the vertical arrows are isomorphisms, the first row and the second row have the same inductive limit, which from the second row is easily seen to be $\Z[\frac{1}{n}]$. Note that $K_0(A_k)\cong\Z$ is generated by the matrix unit $Q_{v_0,k}$, where $v_0$ is the only vertex of $B_n$.

The isomorphism
\begin{equation}\label{eq:varphi}
\varphi:K_0(A)\to\Z[\tfrac{1}{n}]
\end{equation}
induced by the vertical maps in \eqref{eq:indlimit} sends the K-theory generator of $A_k$ to $\lambda^k$, for all $k\geq 0$.

From Prop.~\ref{prop:linebundlesB}, $[Q_{v_0,k}]=[\mathcal{L}_k]$ for all $k\geq 0$, which proves \eqref{eq:LklambdakCuntz} for $k\geq 0$.
From Prop.~\ref{prop:linebundlesA}, $[\mathcal{L}_k]=n^{-k}[1]$ is an integer multiple of $[1]$ for all $k<0$, and so is mapped to $n^{-k}=\lambda^k$ by $\varphi$. This completes the proof of \ref{en:ringCuntzB}.

Point \ref{en:ringCuntzA} immediately follows from Prop.~\ref{prop:generators} or from the above construction of $K_0(C^*(B_n)^{U(1)})$ as an inductive limit.

Finally, $\mathsf{SPic}(C^*(B_n)^{U(1)})$ is a quotient of $\mathbb{N}[t,t^{-1}]$. The quotient map $\pi:\mathbb{N}[t,t^{-1}]\to \mathsf{SPic}(C^*(B_n)^{U(1)})$ sends
$t^k$ to the bimodule class of $\mathcal{L}_k$ for all $k\in\Z$.
Since
\[
\psi\circ\pi(t^k)=\lambda^k \qquad\quad\forall\;k\in\Z,
\]
the composition $\mathbb{N}[t,t^{-1}]\to\Z[1/n]$ is a homomorphism of semirings. Since $\pi$ is a quotient map (of semirings), it follows that $\psi$ is a homomorphism of semirings as well.
\end{proof}

\begin{rem}\label{rem:SMEB}
In the present example, $\mathcal{L}_{-1}$ is free as a left module, but not as a right module.
One can mirror the proof of Prop.~\ref{prop:linebundlesB} and use the right module maps
\[
\mathcal{L}_{-1}\to\mathcal{L}_0 , \; a\mapsto Za , \qquad\text{and}\qquad
\mathcal{L}_0\to\mathcal{L}_{-1} , \; b\mapsto Z^*b ,
\]
to show that $\mathcal{L}_{-1}$ is isomorphic to $ZZ^*\mathcal{L}_0$ as a right module.
Since the class of $[ZZ^*$ is non-trivial in K-theory (it corresponds to $1/n$ under the isomorphism with $\Z[1/n]$), $\mathcal{L}_{-1}$ is not free as a right module.
\end{rem}

\section{The KK-ring of AF cores}\label{sec:7}

One way to overcome the lack of ring structure on the K-theory of a C*-algebra $A$ is to use KK-theory.
To $A$ one can associate the ring $KK(A,A)$, whose multiplication is induced by the internal Kasparov product.
In the Fredholm picture of KK-theory, elements of $KK(A,A)$ are homotopy equivalence classes of pairs $(H,F)$, where $H$ is a $\Z/2$-graded $(A,A)$-bimodules and $F$ a suitable operator (see e.g.~\cite{Bla86} for the details). At the level of bimodules, the internal Kasparov product is induced by the balanced tensor product over $A$.
In the dual picture, line bundles on a compact Hausdorff space $X$ become $(C(X),C(X))$-bimodules, and the ring structure of $K^0(X)$ is also given by the balanced tensor product, the difference being that the bimodules defining $K^0(X)$ are symmetric. In particular, for $A=C(\C P^{n-1})$ the ring \eqref{eq:TPol} is a proper subring of $KK(A,A)$.

\begin{lemma}
Let $A:=C^*(E)^{U(1)}$ and assume that $E$ is finite with an invertible adjacency matrix. Then,
\begin{equation}\label{eq:KKAA}
KK(A,A)\cong\mathrm{Hom}_{\Z}(K_0(A),K_0(A)) .
\end{equation}
\end{lemma}

\begin{proof}
Let $A:=C^*(E)^{U(1)}$.
Finite-dimensional C*-algebras satisfy the UCT (for example because they are Type I C*-algebras), and since the UCT class is closed under inductive limits, all AF-algebras satisfy the UCT (see \cite{BBWW20} for the status of the art about the UCT class).

For UCT algebras \cite[Thm. 7.10]{RS87} there is an isomorphism of rings:
\[
KK(A,A)\cong \bigoplus_{i,j,k=0,1}\mathrm{Ext}_{\Z}^i\big(K_j(A),K_k(A)\big)
\]
where the ring structure on the right hand side is explained in \cite{RS87}.
But $K_0(A)\cong\Z^n$ is free (cf.~Prop.\ref{prop:Gammainverse}) and $K_1(A)=0$.
Since $\mathrm{Ext}^1_{\Z}$ is $0$ for free groups, we get \eqref{eq:KKAA}.
\end{proof}

The isomorphism \eqref{eq:KKAA} maps the class of a pair $(H,F)$ into the endomorphism $H\otimes_A(\;\cdot\;)$.
Since $K_0(A)\cong\Z^n$, with a basis given by vertex projections (cf.~Prop.\ref{prop:Gammainverse}),
\[
KK(A,A)\cong M_n(\Z) ,
\]
with isomorphism sending every endomorphism into its representative matrix relative to such a basis.

On the other hand, every $(M)\in \mathsf{SPic}(A)$ defines an endomorphism $\varphi_L$ of $K_0(A)$ given by
$\varphi_M([N]):=[M\otimes_AN]$, and we have a homomorphism of semirings
\[
\mathsf{SPic}(A)\to KK(A,A) , \qquad M\mapsto\varphi_M .
\]

\begin{prop}
If the matrices $\Gamma$ and \eqref{eq:matrixM} are both invertible, then
\begin{enumerate}
\item\label{en:propMa}
the matrices $1,\Gamma,\Gamma^2,\ldots,\Gamma^{n-1}$ are linearly independent,
\item\label{en:propMb}
in the basis of vertex projections, $\varphi_{\mathcal{L}_1}$ has representative matrix $\Gamma^{-1}$,
\item\label{en:propMc}
the subring of $KK(A,A)$ generated by $\varphi_{\mathcal{L}_1}$ is isomorphic to \eqref{eq:multmor}.
\end{enumerate}
\end{prop}

\begin{proof}
\ref{en:propMa}
By contradiction, suppose that
\[
c_0+c_1\Gamma+\ldots+c_{n-1}\Gamma^{n-1}=0
\]
for some $c_0,\ldots,c_{n-1}\in\C$ not all zero. Then
\[
\sum_{k=0}^{n-1}c_km_{k,j}=\sum_{i=1}^n\Big(\sum_{k=0}^{n-1}c_k\Gamma^{(k)}_{i,j}\big)=0 ,
\]
which contradicts the fact that the matrix \eqref{eq:matrixM} is invertible.

\smallskip

\noindent
\ref{en:propMb}
The representative matrix of $\varphi_{\mathcal{L}_1}$ in the basis of vertex projections can be computed using \eqref{eq:mkagain}.
One has
\[
[P_i]=\sum_{k=0}^{n-1}a_{ik}[\mathcal{L}_k] ,
\]
where $(a_{ik})$ is the inverse matrix of \eqref{eq:matrixM} (columns of this matrix are numbered from $0$ to $n-1$). Thus,
\[
\varphi_{\mathcal{L}_1}([P_i])=[\mathcal{L}_1\otimes_A P_i]=\sum_{k=0}^{n-1}a_{ik}[\mathcal{L}_{k+1}]
\stackrel{\eqref{eq:mkagain}}{=}\sum_{k=0}^{n-1}\sum_{j,h=1}^na_{ik}\Gamma^{(-k-1)}_{h,j} [P_j] .
\]
But
\[
\sum_{k=0}^{n-1}\sum_{h=1}^na_{ik}\Gamma^{(-k-1)}_{h,j}
=\sum_{k=0}^{n-1}\sum_{h,l=1}^na_{ik}\Gamma^{(-k)}_{h,l}\Gamma^{(-1)}_{l,j}=
\sum_{l=1}^n\Big(\sum_{k=0}^{n-1}a_{ik}m_{-k,l}\Big)\Gamma^{(-1)}_{l,j} .
\]
The sum over $k$ gives $\delta_{i,l}$. So,
\[
\varphi_{\mathcal{L}_1}([P_i])==\sum_{j=1}^n \Gamma^{(-1)}_{i,j} [P_j] .
\]

\smallskip

\noindent
\ref{en:propMc}
Because of point \ref{en:propMa}, the minimal polynomial of $\Gamma^{-1}$ is equal to its characteristic polynomial (in other words, $\Gamma^{-1}$ is a \emph{non-derogatory} matrix), so the subring of $M_n(\Z)$ generated by $\Gamma^{-1}$ is isomorphic to \eqref{eq:multmor}.
\end{proof}

\appendix
\section{The Picard group of a finite-dimensional C*-algebra}\label{sec:app}
If $R$ is a commutative ring, isomorphism classes of left (or right) $R$-modules form a semiring with multiplication given by $\otimes_R$, and modules that are invertible w.r.t.~$\otimes_R$ form a group called the \emph{Picard group} of $R$.

The Picard group appears in several places in algebra and geometry.
In number theory, the Picard group of a Dedekind domain $R$ is canonically isomorphic to its ideal class group, and it is trivial if and only if $R$ is a unique factorization domain.
In geometry, if $X$ is a compact Hausdorff space, the Picard group of $C(X)$ is isomorphic to the group of isomorphism classes of line bundles, with operation given by tensor product.

For a commutative ring, one can think of one sided modules as \emph{symmetric} bimodules (i.e.~the left and right actions of $R$ on the bimodule are equal).
If $R$ is a general ring, we define its Picard group $\mathsf{Pic}(R)$ as the group of isomorphism classes of invertible bimodules, \textit{aka} self-Morita equivalence bimodule. Notice that this definition doesn't give the classical Picard group in the commutative case, since not every bimodule over a commutative ring is symmetric. But it is the natural one in the noncommutative case.

In this appendix we collect some remarks about the Picard group of a finite-dimensional $C^*$-algebra.

\medskip

In the rest of this section, we let $A$ be an associative unital algebra over $\C$.
If $L$ is an $A$-bimodule, we denote by $(L)$ its isomorphism class.
Note that the neutral element of $\mathsf{Pic}(A)$ is the class $(A)$ of the trivial bimodule.

To every $\alpha\in\mathrm{Aut}(A)$ we associate an element $({}_\alpha A)\in\mathsf{Pic}(A)$ given by $A$ as a vector space, with right module structure given by the right multiplication, and left module structure given by the left multiplication twisted by $\alpha$:
$$
a\triangleleft m:=\alpha(a)m \;,\qquad\forall\;a,m\in A.
$$
Here ``$\triangleleft$'' is the left module action, while on the right hand side of the equality the multiplication in $A$ is understood. We shall denote by $\mathrm{Inn}(A)$ the inner automorphisms of $A$, and by $\mathrm{Out}(A):=\mathrm{Aut}(A)/\mathrm{Inn}(A)$ the group of outer automorphisms.

\begin{lemma}\label{applemma:1}
Let $\alpha,\beta\in\mathrm{Aut}(A)$. Then
$({}_\alpha A)=({}_\beta A)$ if and only if
$\,\beta\circ\alpha^{-1}\in\mathrm{Inn}(A)$.
\end{lemma}

\begin{proof}
``$\Rightarrow$'' Let $\phi:{}_\alpha A\to {}_\beta A$ be a bimodule isomorphism and $u:=\phi(1)$.
Then
$$
\beta(a)u=\beta(a)\phi(1)=\phi\big(\alpha(a)1\big)=
\phi\big(1\alpha(a)\big)=\phi(1)\alpha(a)=u\alpha(a) \;,
$$
for all $a\in A$.
Since
$$
\phi^{-1}(1)u=\phi^{-1}(u)=1 \qquad\text{and}\qquad
u\phi^{-1}(1)=\phi\big(1\phi^{-1}(1)\big)=1 \;,
$$
$u$ is invertible with inverse $u^{-1}=\phi^{-1}(1)$, and $\beta\circ\alpha^{-1}(a)=u^{-1}au$ is an inner automorphism.

\smallskip

\noindent
``$\Leftarrow$'' Conversely, suppose there esists an invertible element $u\in A$ such that $\beta(a)u=u\alpha(a)$ for all $a\in A$. Then the map $\phi:{}_\alpha A\to {}_\beta A$ defined by $\phi(a):=ua$ is a bimodule isomorphism, with inverse $\phi^{-1}(a)=u^{-1}a$.
\end{proof}

It follows from previous Lemma that there is a well defined map:
\begin{equation}\label{appeq:1}
\mathrm{Out}(A)\to\mathsf{Pic}(A)
\;.
\end{equation}

\begin{lemma}
The map \eqref{appeq:1} is an injective group anti-homomorphism.
\end{lemma}

\begin{proof}
For $\beta=\id$, Lemma \ref{applemma:1} tells us that $({}_\alpha A)=(A)$ if{}f $\alpha$ is inner, hence the map is injective. For every automorphisms $\alpha$ and $\beta$ there is a bimodule map
$$
{}_\alpha A\otimes_{\C} {}_\beta A\to {}_{\beta\circ\alpha} A \;,\qquad
a\otimes b\mapsto \beta(a)b \;.
$$
Since $ab\otimes c-a\otimes\beta(b)c\mapsto 0$ for all $a,b,c\in A$, this induces a bimodule map
${}_\alpha A\otimes_A {}_\beta A\to {}_{\beta\circ\alpha} A$. It is an isomorphism, with inverse given by the bimodule map determined by $1\mapsto 1\otimes_A 1$. Hence 
$({}_\alpha A)\cdot({}_\beta A)=({}_{\beta\circ\alpha} A)$, as claimed.
\end{proof}

\begin{rem}
In general, \eqref{appeq:1} is not surjective. For example, if $A=C(X)$ and $L$ is the  (symmetric) bimodule of continuous sections of a line bundle over $X$, then $(L)=({}_\alpha A)$ only if $\alpha=\id$ (inner automorphisms of a commutative algebra are trivial). But this means that $L\cong A$ is free, i.e.~only classes of trivial line bundles are in the image of the map \eqref{appeq:1}.
\end{rem}

Let $(L)\in\mathsf{Pic}(A)$.
Since $(L)$ is invertible, there is a canonical isomorphism
\begin{equation}\label{appeq:2}
A\to\mathrm{End}_A(L) \;,\qquad
a\mapsto\widehat{a} \;,
\end{equation}
given by $\widehat{a}(m):=a\cdot m\;\forall\;m\in L$.
On the other hand, for every $a\in Z(A)$ one has a right $A$-linear endomorphism $\widehat{\sigma}_L(a)$ defined by:
$$
\widehat{\sigma}_L(a)(m):=ma \;,\qquad\forall\;a\in A,m\in L.
$$
Since $\widehat{\sigma}_L(a)$ is also left $A$-linear, it belongs to the center of $\mathrm{End}_A(L)$, hence it's the image under \eqref{appeq:2} of an element $\sigma_L(a)\in Z(A)$.

\begin{lemma}
$\sigma_L$ is an automorphism of the center of $A$.
\end{lemma}

\begin{proof}
Since $(L)$ is invertible, the map $a\mapsto\{ m\mapsto ma \}$ is an isomorphism $A\to{}_A\mathrm{End}(L)$.
Under the identification of both $\mathrm{End}_A(L)$ and ${}_A\mathrm{End}(L)$ with $A$, the formula $m\cdot\eta(a):=am$ defines an endomorphism $\eta$ of $Z(A)$ that is the inverse of the endomorphism $\sigma_L$.
\end{proof}

\begin{lemma}
$\sigma_L\in\mathrm{Aut}(Z(A))$ depends only on the class of $(L)\in\mathsf{Pic}(A)$.
\end{lemma}

\begin{proof}
Let $(L)=(L')\in\mathsf{Pic}(A)$ and $\varphi:L\to L'$ a bimodule isomorphism. Applying $\varphi$ to both sides of
$$
\sigma_L(a)m=\widehat{\sigma}_L(a)(m)=ma
$$
we get
$$
\varphi\big(\sigma_L(a)m\big)=\varphi(m)a=
\widehat{\sigma}_{L'}(a)\big(\varphi(m)\big)=
\sigma_{L'}(a)\varphi(m) \;.
$$
Since $\varphi$ is also a left module map, it follows that $\sigma_L(a)\varphi(m)=\sigma_{L'}(a)\varphi(m)$ for all $m$, i.e.~$\sigma_L(a)=\sigma_{L'}(a)$ for all $a$.
\end{proof}

\begin{lemma}
The map
\begin{equation}\label{appeq:3}
\mathsf{Pic}(A)\to\mathrm{Aut}(Z(A))
\;,\qquad
(L)\mapsto \sigma_L \;,
\end{equation}
is a group homomorphism.
\end{lemma}

\begin{proof}
Let $(L),(L')\in\mathsf{Pic}(A)$. Then, for all $a\in A$, $m\in L$, $m'\in L'$:
\begin{multline*}
\sigma_{L\otimes_AL'}(a)m\otimes_Am'=m\otimes_A m'a=
m\otimes_A\sigma_{L'}(a)m' \\
=m\sigma_{L'}(a)\otimes_Am' =
\sigma_L\big(\sigma_{L'}(a)\big)m\otimes_Am' \;.
\end{multline*}
Therefore, $\sigma_{L\otimes_AL'}=\sigma_L\circ\sigma_{L'}$.
\end{proof}

The kernel $\mathsf{KPic}(A)$ of \eqref{appeq:3} is called the \emph{static Picard group} in \cite{BW05}.
In general, $(L)\in\mathsf{KPic}(A)$ if (by definition):
$$
am=ma \;,\qquad\forall\;a\in Z(A),m\in L.
$$
Thus, as a $Z(A)$-$Z(A)$ bimodule, $L$ is a symmetric bimodule.

\begin{rem}
For $A=C(X)$, $\mathsf{KPic}(A)$ is the classical Picard group (given by classes of symmetric invertible bimodules, i.e.~line bundles), and $\mathsf{Pic}(A)$ the semidirect product of $\mathsf{KPic}(A)$ and $\mathrm{Aut}(A)$.
Notice that $\mathrm{Aut}(A)$ is isomorphic to the group of homeomorphisms $X\to X$.
\end{rem}

\begin{rem}
If $A$ is any commutative algebra, then $\mathsf{KPic}(A)$ is abelian. Indeed, for all symmetric bimodules $L,L'$, the flip
$L\otimes_AL'\to L'\otimes_AL$, $m\otimes_Am'\mapsto m'\otimes_A m$,
is a well defined bimodule isomorphism.
\end{rem}

In the rest of this section we will prove the following theorem:

\begin{thm}
For any finite-dimensional $C^*$-algebra $A$, $\mathsf{Pic}(A)\cong\mathrm{Aut}(Z(A))$.
\end{thm}

\begin{proof}
Let $A$ be a finite-dimensional $C^*$-algebra. Since isomorphic algebras have isomorphic Picard groups, without loss of generality we can assume that:
$$
A=\bigoplus\nolimits_{i=1}^nM_{d_i}(\C) \;,
$$
for some positive integers $n,d_1,\ldots,d_n$.

Let $P_i\in M_{d_i}(\C)$ be the unit of the $i$-th summand in $A$. Note that
$$
Z(A)=\C P_1\oplus\C P_2\oplus\ldots\oplus\C P_n
$$
and $\mathrm{Aut}(Z(A))\cong\mathfrak{S}_n$ is the group of permutations $n$ elements (the coefficients of the projections $P_1,\ldots,P_n$).
To any permutation $\tau\in\mathfrak{S}_n$ we can associate a bimodule:
\begin{equation}\label{appeq:4}
L:=\bigoplus_{i=1}^sL_i:=\bigoplus_{i=1}^s M_{d_i\times d_{\tau(i)}}(\C) \;.
\end{equation}
For $a=(a_1,\ldots,a_n)\in A$ and $m=(m_1,\ldots,m_n)\in L$, left and right actions are defined in terms of matrix multiplication as follows:
\begin{align*}
am &:=(a_1m_1,a_2m_2,\ldots,a_nm_n) \;, \\
ma &:=(m_1a_{\tau(1)},m_2a_{\tau(2)},\ldots,m_na_{\tau(n)}) \;.
\end{align*}
Now, let $a$ be central, that is $a_i=\lambda_iP_i$ with $\lambda_i\in\C$ for all $i=1,\ldots,n$.
Then
$$
\sigma_L(a)m=(\lambda_{\tau(1)}m_1,\ldots,\lambda_{\tau(n)}m_n)=(\lambda_{\tau(1)}P_1+\ldots+\lambda_{\tau(n)}P_n)m \;.
$$
Under the above identification of $\mathrm{Aut}(Z(A))$ with $\mathfrak{S}_n$, $\sigma_L=\tau$ is exactly the permutation we started from, proving that the map \eqref{appeq:3} is surjective.

\medskip

It remains to prove that it is also injective, i.e.~that $\mathsf{KPic}(A)$ is trivial.
Every $(H)\in\mathsf{Pic}(A)$ is represented by an $A$-bimodule $H$ that is finitely generated projective as one-sided module. Since $A$ is finite-dimensional, $H$ must be finite-dimensional as well.

Let $H$ be any finite-dimensional left $A$-module. Since $P_1,\ldots,P_s$ are orthogonal projections, $H$ decomposes as
$H=\bigoplus_{i=1}^nH_i$, with
$$
H_i=P_i\cdot H\cong M_{d_i\times k_i}(\C) \;.
$$
Here $k_i$ is the multiplicity of the (unique) irreducible representation $\C^{d_i}$ of $M_{d_i}(\C)$ in $P_i\cdot H$, and the action of the algebra $M_{d_i}(\C)$ on $M_{d_i\times k_i}(\C)$ is given by left matrix multiplication.

Let $f\in{}_A\mathrm{Hom}(H_i,H)$ be a left $A$-linear map. Then $P_jf(v)=f(P_jv)=0$ for all $v\in H_i$ if $j\neq i$. Thus $f(H_i)\subseteq H_i$ and 
$$
\mathrm{End}_A(H)\cong\bigoplus\nolimits_{i=1}^n\mathrm{Hom}_A(H_i,H)\cong
\bigoplus\nolimits_{i=1}^n\mathrm{End}_A(H_i) \;.
$$
Therefore,
$$
\mathrm{End}_A(H)\cong\bigoplus\nolimits_{i=1}^nM_{k_i}(\C)
$$
where the $i$-th summand act on $H_i\cong M_{n_i\times k_i}(\C)$ via matrix multiplication from the right. One has $\mathrm{End}_A(H)\cong A$, i.e.~$(H)\in\mathsf{Pic}(A)$, if and only if there exists a permutation $\tau\in \mathfrak{S}_n$ such that
$$
(d_{\tau(1)},\ldots,d_{\tau(n)})=(k_1,\ldots,k_n) \;.
$$
Every element in $\mathsf{Pic}(A)$ is then represented by a bimodule of the form \eqref{appeq:4}.
Since $mP_i=P_{\tau(i)}m\;\forall\;m$, the bimodule is symmetric if and only if $\tau=\id$, that means that $(H)=(A)$ is trivial. Thus $\mathsf{KPic}(A)=\{(A)\}$
is contains only the neutral element.
\end{proof}

\end{document}